\documentclass[a4paper,12pt]{amsart}
\usepackage{lineno}
\usepackage{epsfig}
\usepackage{subfig}
\usepackage{epstopdf}
\usepackage{rotating}
\usepackage[usenames]{color}
\usepackage[dvipsnames]{xcolor}
\usepackage{multirow}
\usepackage{tikz}
\usetikzlibrary{cd}
\usetikzlibrary{calc}
\usepackage{verbatim}
\usepackage{multicol}
\usepackage{cite}
\usepackage{soul}
\usepackage{amsthm,mathrsfs}
\usepackage{thmtools,thm-restate} 
\usepackage{url}
\usepackage[english]{babel}
\usepackage[utf8]{inputenc}
\usepackage[margin=1in]{geometry}
\usepackage[colorinlistoftodos]{todonotes}
\usepackage{dsfont}
\usepackage{xparse}
\usepackage{amssymb, amsmath, amscd, amsthm}
\usepackage{amsfonts}
\usepackage{psfrag}
\usepackage[all]{xy}
\usepackage{graphicx}
\usepackage{verbatim}
\usepackage{mathtools}
\usepackage{faktor}
\usepackage[normalem]{ulem}
\usepackage{xfrac}
\usepackage{cancel}

\newtheorem{theorem}{Theorem}[section]
\newtheorem{lemma}[theorem]{Lemma}

\theoremstyle{definition}
\newtheorem{definition}[theorem]{Definition}
\newtheorem{example}[theorem]{Example}

\theoremstyle{remark}
\newtheorem{remark}[theorem]{Remark}
\numberwithin{equation}{section}

\newtheorem{proposition}[theorem]{Proposition}
\newtheorem{corollary}[theorem]{Corollary}

\newcommand{\R}{\mathbb{R}}

\newcommand{\N}{\mathbb{N}}

\newcommand{\id}{\mathrm{id}}

\newcommand{\vertiii}[1]{{\left\vert\kern-0.25ex\left\vert\kern-0.25ex\left\vert #1
\right\vert\kern-0.25ex\right\vert\kern-0.25ex\right\vert}}

\makeatletter
\@namedef{subjclassname@2020}{\textup{2020} Mathematics Subject Classification}
\makeatother

\begin{document}

\title[On the geometric and Riemannian structure of the spaces of GENEOs]{On the geometric and Riemannian structure of the spaces of group equivariant non-expansive operators}

\author{Pasquale Cascarano}
\address{Dipartimento di Matematica, Universit\`a di Bologna, Italia}
\curraddr{}
\email{pasquale.cascarano2@unibo.it}

\author{Patrizio Frosini}
\address{Dipartimento di Matematica, Universit\`a di Bologna, Italia}
\curraddr{}
\email{patrizio.frosini@unibo.it}

\author{Nicola Quercioli}
\address{Dipartimento di Matematica, Universit\`a di Bologna, Italia}
\curraddr{}
\email{nicola.quercioli2@unibo.it}

\author{Amir Saki}
\address{School of Mathematics, Institute for Research in Fundamental Sciences (IPM), Iran}
\curraddr{}
\email{amir.saki.math@gmail.com}

\subjclass[2020]{Primary: 55N31; Secondary: 58D20, 58D30, 62R40, 65D18, 68T09}


\date{}

\dedicatory{}


\begin{abstract}
Group equivariant non-expansive operators have been recently proposed as basic components in topological data analysis and \textcolor{black}{ geometric deep learning}. In this paper we study some geometric properties of the spaces of group equivariant operators and show how a space $\mathcal{F}$ of group equivariant non-expansive operators can be endowed with the structure of a Riemannian manifold, making available the use of gradient descent methods for the minimization of cost functions on $\mathcal{F}$. \textcolor{black}{ Our mathematical model takes into account the probability distribution defined on the data space.} As an application of this approach, we also describe a procedure to select a finite set of representative group equivariant non-expansive operators in the considered manifold.
\end{abstract}

\maketitle


\section{Introduction}
\label{intro}

The concept of group equivariant non-expansive operator (GENEO) has been proposed as a versatile tool in topological data analysis and deep learning \cite{Fr16,FJ16,BFGQ19}, with a particular reference to \textcolor{black}{geometric deep learning \cite{BrBrLeal17}}. We recall that an operator $F$ is called equivariant with respect to a group $G$ if the action of $G$ commutes with $F$, \textcolor{black}{while $F$ is called non-expansive if it does not increase the distance between data.} The study of these operators appears to be interesting for several reasons. On the one hand, the use of GENEOs allows to formalize the role of the observer in data analysis and data comparison.
\textcolor{black}{This is an important goal as it appears clearer and clearer that an efficient approach to these fields of research requires examining not only the data but also the observer that analyzes them. Indeed, the observer's role is often more important than the role of the data themselves. On the other hand, GENEOs could be of help in building new deep learning architectures because of their modularity, and in particular of the fact that the convex combination, composition, maximum, and minimum of GENEOs are still GENEOs \cite{BFGQ19}.}

Several researchers have pointed out that the use of \textcolor{black}{group equivariant operators (GEOs)} could be of great relevance in the development of machine learning \textcolor{black}{ \cite{AnRoPo16,cohen2016group,worrall2017harmonic,GeArCaal23}}, since these operators can inject pre-existing knowledge into the system and allow us to increase our control on the construction of neural networks \cite{Bengio2013}.
Further, it is also worth noting that invariant and non-expansive operators can be used to reduce data variability \cite{Ma12,Ma16}, and that in the last years equivariant transformations have been studied for learning symmetries \cite{ZVERP15,AERP19}.

\textcolor{black}{ Availability of new techniques to build GENEOs allows for new applications in data analysis and machine learning. For instance, \textcolor{black}{ a learning paradigm called GENEOnet has been successfully used for detecting pockets on the surface of proteins hosting ligands \cite{Mi23}}.
Also, advances in the topological and geometric theory of GENEOs can improve our ability in building suitable operators for new applications. For instance, in \cite{Bocchi2023}, in finite settings, all linear GENEOs with respect to a transitive acting group have been characterized.
Further, in \cite{10.3389/frai.2022.786091}, a method for defining non-linear GENEOs using symmetric functions and \textit{permutants} has been provided,
while in \cite{AhFeFr23} it has been shown that these operators can be used to compare graphs, and in \cite{FeFrQuTo23} GENEOs have been adapted to the case of partial equivariance.
}
\textcolor{black}{ Finally, the research concerning the representation of observers as operators could help us in understanding the role of conflicts and contradictions in the development of intelligence} \cite{Fro09}.

This focus on spaces of group equivariant non-expansive operators stresses the need for a study of the topological and geometric structure of these spaces, in order to simplify their exploration and use.

\textcolor{black}{To deal with the real-world problems, it is important to extend the theory of GENEOs to the case of signals endowed with probability distributions, which is the main motivation of this paper. For example, the use of GENEOs in the analysis of \textcolor{black}{ noisy data requires taking probabilities into account \cite{FrGrPa23}}.}
Hence, extending the mathematical framework proposed in \cite{BFGQ19}, in this paper we show how a space of group equivariant non-expansive operators can be endowed with a suitable metric that takes into account the different probability of data. With reference to such a metric, we prove that when the spaces of data are compact, the space of GENEOs is compact too, so implying that it admits arbitrarily good approximations by finite sets of operators.
Furthermore, we show that we can define a suitable Riemannian structure on finite dimensional submanifolds of the space of GENEOs and use this structure to approximate those manifolds by minimizing a cost function via gradient descent methods. \textcolor{black}{Indeed, to find optimal GENEOs, we have to minimize cost functions. This task can be made easier if we can consider the gradients of functions defined on the space of GENEOs. This justifies the last goal of our paper (i.e., endowing the space of GENEOs with a Riemannian structure).
}
{The Riemannian structure we propose is based on the comparison of the action of GENEOs on data, taking into account the probability distribution on the set of admissible signals we are considering.}
At the end of the paper, we describe a computational experiment, where we use our approach to select a representative set of GENEOs from a given manifold of operators.

{From our perspective, endowing the spaces of operators/observers acting on the data with new topological and geometric structures is a key point in the present research concerning the mathematical aspects of machine learning. New results in this direction could allow exploiting GENEOs as components in the realization of new kinds of neural networks, taking advantage from the dimensionality reduction offered by Topological Data Analysis \cite{BFGQ19}.}

The outline of the paper is as follows. In Section \ref{MS} we illustrate our mathematical setting and prove its main properties.
In Section \ref{Experiments} we show how our framework can be used to find a good finite approximation of the considered manifold of GENEOs. In Section \ref{Discussion} we conclude the paper with a brief discussion.

\section{Mathematical setting}
\label{MS}
Since we are interested in studying operators that transform signals, we have to establish which signals will be admitted and which structures will be considered on them. The signals we will consider will be represented as bounded real-valued functions on a set $X$ belonging to \textcolor{black}{a subset $\Phi$ of a finite dimensional vector space $V$}. This vector space will be endowed with a suitable norm and inner product. Furthermore, our \textcolor{black}{vector space $V$} will be associated with a probability distribution that describes how frequently each signal occurs. Then, we \textcolor{black}{will} explore the properties  of our framework to show its natural and suitable mathematical behavior towards the definition of a Riemannian manifold structure on the space of GENEOs. Indeed, we \textcolor{black}{will} show how the norm on signals induce a pseudo-metric on $X$, in such a way that all signals are continuous with respect to this pseudo-metric. \textcolor{black}{Then, we will focus on the permutations of $X$ that preserve our space of data and show that these permutations are isometries of $X$ in our model, so describing the symmetries in the space of admissible signals. For the remainder of the paper, we will fix a subgroup $G$ of the group of such permutations, and we will show that $G$ admits the structure of a topological group.}  Finally, we \textcolor{black}{will} discuss about GEOs and GENEOs, and explore some of their properties. Especially, we \textcolor{black}{will} show that under some assumptions, the space of GENEOs is convex and compact. Also, we \textcolor{black}{will} show how the group $G$ acts on the space of all GENEOs defined between two sets of signals. Furthermore, we \textcolor{black}{will} explain how we can implement a Riemannian structure on a manifold of GENEOs.

{\color{black}
\subsection{Probability inner product $G$-spaces and spaces of admissible signals.}
\label{frame}
Let us list some basic assumptions and definitions, which will be used throughout the paper:

\begin{enumerate}
    \item We choose a set $X$ and an $n$-dimensional vector space $V\subseteq \mathbb{R}^X_b$, where $\mathbb{R}^X_b$ is the set of all bounded functions from $X$ to $\mathbb{R}$. $V$ is the vector space where the functions representing our signals/data belong to.
    \item We choose a subgroup $G$ of the group $\mathrm{Aut}_{\Phi}(X)$ of all bijections $g\colon X\to X$. We observe that $\mathrm{Aut}_\Phi(X)$ and $G$ naturally act on $\mathbb{R}^X_b$ by composition on the right.
        For any $g\in G$, we consider the map $\rho_g:\mathbb{R}^X_b\to \mathbb{R}^X_b$ taking each $\varphi\in \mathbb{R}^X_b$ to $\varphi g$.
    \item \label{assumptionip} We assume that an inner product $\langle\cdot,\cdot\rangle_V$ on $V$ is given, which is invariant under the action of $G$ (i.e., $\langle\rho_g(\varphi_1),\rho_g(\varphi_2)\rangle_V= \langle \varphi_1,\varphi_2 \rangle_V$ for any $\varphi_1,\varphi_2\in V$ and $g\in G$).
    \item \label{coeffalpha} We denote the norm arising  from $\langle\cdot,\cdot \rangle_V$ by $\lVert\cdot\rVert_V$.
        Moreover, for any $\varphi\in V$, we also consider its $L^{\infty}$-norm $\lVert\varphi\rVert_{\infty}=\sup_{x\in X}|\varphi(x)|<\infty$.
        Assumption (\ref{assumptionip}) immediately implies that $\lVert \cdot \rVert_V$ is invariant under the action of $G$, i.e., $\|\rho_g(\varphi)\|_V=\|\varphi\|_V$ for any $\varphi\in V$ and $g\in G$. We note that $\lVert\cdot\rVert_V$ and $\lVert\cdot\rVert_{\infty}$ are equivalent, since $V$ is finite dimensional, i.e., there exist two real numbers $\alpha,\beta>0$ such that $\alpha\lVert\cdot\rVert_{\infty}\le\lVert\cdot\rVert_V\le\beta\lVert\cdot\rVert_{\infty}$. Hence, $\lVert\cdot\rVert_V$ and $\lVert\cdot\rVert_{\infty}$ induce the same topology $\tau_V$ on $V$.
    \item Finally, we assume that there exist a compact subspace $\Phi$ of $V$, representing the data of interest, and a probability Borel measure $\mu$ on $V$, such that
    $(i)$ $\mu$ is invariant under the action of $G$ (i.e., $\mu(A)=\mu(\rho_g(A))$ for every Borel set $A\subseteq V$ and every $g\in G$), and $(ii)$ the support of $\mu$ coincides with $\Phi$.
    We recall that the support of $\mu$ is defined to be the closure of the set of all points $\varphi\in V$ for which every open neighborhood of $\varphi$ has a positive measure.
    For every Borel set $A\subseteq V$, $\mu(A)$ represents the probability that our measurement produces a signal belonging to $A$.
    To facilitate our exposition, we will consider a {\color{black}finite} Borel measure $\nu$ on $V$ and a {non-negative} integrable function $f:V\to \R$, such that $\mu(A)=\int_A f d\nu$ for any Borel set $A\subseteq V$.
    As a consequence, $\mu$ is absolutely continuous with respect to $\nu$ and $f$ is the Radon-Nikodym derivative of $\mu$ with respect to $\nu$. We will also require that $f$ and $\nu$ are invariant under the action of $G$ (i.e., if $g\in G$ then $f(\rho_g(\varphi))=f(\varphi)$ for any $\varphi\in \Phi$, and {$\nu(A)=\nu(\rho_g(A))$ for every Borel set $A\subseteq V$};
    e.g., we can choose $f$ equal to the indicator function of $\Phi$ and $\nu=\mu$).
     Of course, $\mu(\Phi)=\int_\Phi f d\nu=1$, since $\mu$ is a probability measure and $\Phi$ is the support of $\mu$. Therefore, $\mu$ can be also seen as a probability Borel measure on $\Phi$. With a small abuse of notation, we will still use the symbols $\mu$ and $\nu$ to denote the restrictions of these Borel measures to $\Phi$, and $f$ to denote the restriction of $f$ to $\Phi$.
\end{enumerate}

We stress the following remarks:
\begin{description}
  \item[$(a)$] Point (1) implies that the vector space $V$ is isomorphic to $\mathbb{R}^n$.
  \item[$(b)$] Since $\|\varphi_1 g-\varphi_2 g\|_\infty=\|\varphi_1-\varphi_2\|_\infty$,
the maps $\rho_g$, $\rho_{g^{-1}}=({\rho_g})^{-1}$ are homeomorphisms with respect to $\|\cdot\|_\infty$ (and hence also with respect to $\|\cdot\|_V$). Hence, if $A$ is a Borel set in $V$, then $\rho_g(A)$ is a Borel set in $V$.
  \item[$(c)$]  Since $\Phi$ is the support of the measure $\mu$, which is invariant under the action of $G$, it is easy to check that $\varphi g\in\Phi$ for any $\varphi\in\Phi$ and any $g\in G$. In other words, the elements of the group $G$ represent some invariances of the signal space $\Phi$.
\end{description}

Under the previous assumptions (1)-(5), we  say that $\Phi$ is a space of \emph{admissible signals} in probability, and $V$ is a (finite dimensional)
\textcolor{black}{\emph{probability inner product $G$-space}.}

\begin{example}
\label{basicexample}
A simple example illustrating the concepts of probability inner product $G$-space and space of admissible signals
in probability can be given by considering a set of grayscale images on the square $X=\{(x,y)\in\R^2:|x|,|y|\le 1\}$, as explained in the following.
Here, we set $V$ equal to the collection of all functions $\varphi$ that can be obtained by restricting to $X$ some polynomial $p(x,y)=\sum_{0\le i,j\le k-1}a_{ij} x^iy^j$ of degree strictly less than $k$ in each variable. The vector space $V$ is isomorphic to $\R^{k^2}$.
$G$ is set equal to the group of all rigid motions preserving the square $X$.
In other words, $G$ is generated by the map $q_1$ taking $(x,y)$ to $(-x,y)$, and the map $q_2$ taking $(x,y)$ to $(y,x)$.
If we identify $V$ and $\R^{k^2}$, the inner product $\langle\cdot,\cdot \rangle_V$ can be chosen equal to the standard inner product on $\R^{k^2}$, while the probability measure $\mu$ can be defined by setting $\mu(A):=\frac{\lambda(A\cap\Phi)}{\lambda(\Phi)}$ for every Borel set $A\subseteq V$,
where $\Phi=\{p\in V: p(X)\subseteq [0,1]\}\subseteq V$ and $\lambda$ is the Lebesgue measure on $\R^{k^2}$.
We observe that, if $p\in V$ and $p(x,y)=\sum_{0\le i,j\le k-1}a_{ij} x^iy^j$, then $pq_1(x,y)=\sum_{0\le i,j\le k-1}(-1)^i a_{ij} x^iy^j$,
while $pq_2(x,y)=\sum_{0\le i,j\le k-1}a_{ji} x^iy^j$.
It follows that the actions of $q_1$ and $q_2$ on $V$ by composition on the right preserve both the inner product $\langle\cdot,\cdot \rangle_V$ and the probability measure $\mu$.
In this example, $\Phi$ can be interpreted as a set of grayscale images on $X$, where the values $0$ and $1$ correspond to black and white, respectively, while the values between $0$ and $1$ represent different levels of gray.
With these choices, we can easily check that the previous assumptions (1)-(5) hold. Therefore, $V$ is a probability inner product $G$-space and $\Phi$ is a space of admissible signals in probability.
\end{example}
}

\subsection{The pseudo-metric on $X$ induced by admissible signals.}

In this subsection, we show that the distance {\color{black} $\Delta_{\Phi}$} on $\Phi$, induced by $\lVert\cdot\rVert_V$, can be pulled back to a pseudo-metric on $X$.
We recall that a pseudo-metric on a set $X$ is  a distance $d$ without the property $d(x,y)=0\implies x=y$. Note that the collection of all open balls $B_X(x,r)=\{y\in X:d(x,y)<r\}$ is a basis for a topology on $X$, {which is called the \emph{topology induced by $d$}.}

The following lemma prepares us to define a pseudo-metric on $X$.

\begin{lemma}\label{Lemma1}
Let $x_1,x_2\in X$. Then, the map $\xi_{x_1,x_2}:\Phi\to\mathbb{R}$ defined by $\xi_{x_1,x_2}(\varphi)=|\varphi(x_1)-\varphi(x_2)|$ for any $\varphi\in\Phi$, is continuous with respect to \textcolor{black}{the topology $\tau_\Phi$ on $\Phi$ induced by the topology $\tau_V$ on $V$,} and it is an integrable random variable with respect to $\nu$.
\end{lemma}

\begin{proof}
Let $\varphi_1,\varphi_2\in \Phi$. Then,
\begin{align*}
|\xi_{x_1,x_2}(\varphi_1)-\xi_{x_1,x_2}(\varphi_2)|&=\big|\left|\varphi_1(x_1)-\varphi_1(x_2)\right|-\left|\varphi_2(x_1)-\varphi_2(x_2)\right|\big|\\
&\le|(\varphi_1(x_1)-\varphi_1(x_2))-(\varphi_2(x_1)-\varphi_2(x_2))|\\
&=|(\varphi_1(x_1)-\varphi_2(x_1))-(\varphi_1(x_2)-\varphi_2(x_2))|\\
&\le|\varphi_1(x_1)-\varphi_2(x_1)|+|\varphi_1(x_2)-\varphi_2(x_2)|\\
&\le 2\lVert\varphi_1-\varphi_2\rVert_{\infty}
\le \frac{2}{\alpha}\lVert\varphi_1-\varphi_2\rVert_V
\end{align*}
\textcolor{black}{
where $\alpha$ is the coefficient defined in point (\ref{coeffalpha}) of Subsection \ref{frame}.}
Therefore, $\xi_{x_1,x_2}$ is continuous with respect to \textcolor{black}{the topology $\tau_\Phi$ induced by the topology $\tau_V$ defined in point (\ref{coeffalpha}) of Subsection \ref{frame}}.
\textcolor{black}{Now, since $\xi_{x_1,x_2}$ is continuous on the compact space $\Phi$, it is bounded. \textcolor{black}{Moreover, the counterimage $\xi_{x_1,x_2}^{-1}(U)$ of each open set $U\subseteq \R$ is an open set (and hence a Borel set) in $\Phi$.
It follows that the counterimage of each Borel set is a Borel set. This proves that the function $\xi_{x_1,x_2}$ is measurable, since the restriction of $\nu$ to $\Phi$ is still a Borel measure.}
Hence, \textcolor{black}{recalling that the measure $\nu$ is finite}, if $M$ is \textcolor{black}{an upper bound for $|\xi_{x_1,x_2}|$}, then we have that
\[\int_{\Phi}\left|\xi_{x_1,x_2}(\varphi)\right|\,d\nu\le M\nu(\Phi)<\infty.\] }
\end{proof}

Now, we define a pseudo-metric $\Delta_X$ on $X$ as follows:
$$\Delta_X(x_1,x_2)=\int_{\Phi}|\varphi(x_1)-\varphi(x_2)|f(\varphi)\,d\nu,\quad \forall\,x_1,x_2\in X.$$
In plain words, the distance between two points $x_1,x_2$ is set to be the expected value of the function $\xi_{x_1,x_2}$ \textcolor{black}{with respect to the probability measure $\mu$}.

Another relevant pseudo-metric on $X$ is the pseudo-metric $D_X$ defined by setting
$$D_X(x_1,x_2):= \max_{\varphi \in \Phi}|\varphi(x_1) - \varphi(x_2)|.$$
The definition of $D_X$ is well-posed, since $\Phi$ is compact.

\textcolor{black}{
We observe that
\begin{align*}
 \Delta_X(x_1,x_2)
 &=\int_{\Phi}|\varphi(x_1)-\varphi(x_2)| f(\varphi)\,d\nu \\
 &\le \int_{\Phi} D_X(x_1,x_2) f(\varphi)\,d\nu \\
 &= D_X(x_1,x_2)\int_{\Phi} f(\varphi)\,d\nu\\
 &= D_X(x_1,x_2)
 \end{align*}
for every $x_1,x_2\in X$.
The inequality $\Delta_X\le D_X$ implies that the topology $\tau_{D_X}$ induced by $D_X$ is finer than the topology $\tau_{\Delta_X}$ induced by $\Delta_X$.
We recall that if $(X, D_X)$ is complete, {then it is compact \cite[Theorem 1]{BFGQ19}.}
}

\textcolor{black}{For the remainder of the paper}, whenever not differently specified, we assume that $X$ is equipped with the topology arising from $\Delta_X$.
\begin{proposition}\label{pp2}
Each function $\varphi_0\in\Phi$ is continuous with respect to $\Delta_X$.
\end{proposition}

\begin{proof}
\textcolor{black}{
Let $x_0\in X$. After choosing an $\varepsilon>0$, let us consider the ball $$B_\varepsilon
=\left\{\varphi\in\Phi:\lVert\varphi_0-\varphi\rVert_{\infty}\le\frac{\varepsilon}{4}\right\}.$$
\textcolor{black}{Since $\varphi_0$ is in the support of $\mu$,} $\mu(B_\varepsilon)$ is positive.
Hence, for every $x\in X$ and $\varphi\in B_{\epsilon}$,}
\begin{align*}\label{ineq1}
|\varphi(x)-\varphi(x_0)|&=|(\varphi_0(x)-\varphi_0(x_0))+(\varphi(x)-\varphi_0(x))+(\varphi_0(x_0)-\varphi(x_0))|\\
&\ge |\varphi_0(x)-\varphi_0(x_0)|-|\varphi(x)-\varphi_0(x)|-|\varphi_0(x_0)-\varphi(x_0)|\\
&\ge |\varphi_0(x)-\varphi_0(x_0)|-2\lVert\varphi_0-\varphi\rVert_{\infty}\nonumber\\
&\ge |\varphi_0(x)-\varphi_0(x_0)|-\varepsilon/2\nonumber.
\end{align*}
This implies that for every $x\in X$,
\begin{align*}
\Delta_X(x,x_0)&=\int_{\Phi}|\varphi(x)-\varphi(x_0)|f(\varphi)\,d\nu\\ &
\ge\int_{B_\varepsilon}|\varphi(x)-\varphi(x_0)|f(\varphi)\,d\nu\\
&\ge \mu(B_\varepsilon) (|\varphi_0(x)-\varphi_0(x_0)|-\varepsilon/2).
\end{align*}
It follows that \textcolor{black}{$|\varphi_0(x)-\varphi_0(x_0)|\le\frac{\Delta_X(x,x_0)}{\mu(B_\varepsilon)}+\varepsilon/2$}.
Therefore, if $\Delta_X(x,x_0)\le\left(\varepsilon/2\right)\mu(B_\varepsilon)$, then
$|\varphi_0(x)-\varphi_0(x_0)|\le\varepsilon$.
\end{proof}

\begin{proposition}
\label{Phi-continuous}
Every function $\varphi\in\Phi$ is non-expansive, and hence continuous, with respect to $D_X$.
\end{proposition}
\begin{proof}
If $x_1,x_2\in X$, then for any function $\varphi \in \Phi$ we have that
$|\varphi(x_1)-\varphi(x_2)| \le \max_{\varphi \in \Phi}|\varphi(x_1)-\varphi(x_2)|=D_X(x_1,x_2)$.
\end{proof}
We now recall that the initial topology $\tau_\text{in}$ on $X$ with respect to $\Phi$ is the coarsest topology on $X$ such that each function $\varphi$ in $\Phi$ is continuous.
From Proposition \ref{pp2} and the definition of $\tau_\text{in}$, it follows that $ \tau_\text{in} \subseteq \tau_{\Delta_X}$.
{We already know that $\tau_{\Delta_X}\subseteq \tau_{D_X}$.}
Since $\Phi$ is compact, we know that $\tau_\text{in}$ and $ \tau_{D_X}$ are the same (see Theorem 2.1 in \cite{BFGQ19}). Hence, $\tau_{D_X} = \tau_{\Delta_X} = \tau_\text{in}$.

Before proceeding, we recall the following lemma (see \cite{Ga64}):
\begin{lemma}
	\label{TOTBOUNDLEM}
	Let $(P,d)$ be a pseudo-metric space. The following conditions are equivalent:
	\begin{enumerate}
		\item $P$ is totally bounded;
		\item every sequence in $P$ admits a Cauchy subsequence.
	\end{enumerate}
	
\end{lemma}
The definition of the pseudo-metric $D_X$ on $X$ relies on $\Phi$. Thus, properties on $\Phi$ naturally induce properties on $X$, as shown in the proof of the following statement.
\begin{proposition}
	\label{thmXTOTBOUND}
	$X$ is totally bounded with respect to $D_X$ and $\Delta_X$.
\end{proposition}

\begin{proof}
{Because of Lemma \ref{TOTBOUNDLEM}, in order to prove that $X$ is totally bounded with respect to $D_X$ we have just to prove that every sequence $(x_i)_{i\in \mathbb{N}}$ in $X$ admits a Cauchy subsequence with respect to $D_X$.}
Let us consider an arbitrary sequence $(x_i)_{i \in \mathbb{N}}$ in $X$ and an arbitrarily small $\varepsilon > 0$. Since $\Phi$ is totally bounded, we can find a finite subset $\Phi_{\varepsilon} = \{ \varphi_1, \dots , \varphi_n \}$ of $\Phi$, such that $\Phi = \bigcup_{i=1}^n B_{\Phi}(\varphi_i, \varepsilon)$, where
$B_{\Phi}(\varphi,\varepsilon)=\{\varphi' \in \Phi : {\lVert \varphi'-\varphi\rVert_{\infty} }< \varepsilon \}$.
As a consequence, if $\varphi \in \Phi$, then there exists $\varphi_{\bar k} \in \Phi_\varepsilon$ such that $\|\varphi - \varphi_{\bar k}\|_{\infty} < \varepsilon$.
Now, we consider the real sequence $(\varphi_1(x_i))_{i\in \mathbb{N}}$, \textcolor{black}{ which} is bounded in $\mathbb{R}$ because all the functions in $\Phi$ are bounded. From the Bolzano-Weierstrass Theorem
it follows that we can extract a Cauchy subsequence $\left(\varphi_1(x_{i_h})\right)_{h\in \mathbb{N}}$. Then we consider the sequence
$\left(\varphi_2(x_{i_h})\right)_{h\in \mathbb{N}}$. Since $\varphi_2$ is bounded, we can extract a Cauchy subsequence $\left(\varphi_2(x_{i_{h_t}})\right)_{t\in \mathbb{N}}$.
We can repeat the same argument for any $\varphi_k \in \Phi_{\varepsilon}$.
Thus, we obtain a subsequence $(x_{p_{j}})_{j\in \mathbb{N}}$ of $(x_i)_{i\in \mathbb{N}}$, such that $\left(\varphi_k(x_{p_j})\right)_{i\in \mathbb{N}}$ is a Cauchy sequence
for any $k \in \{1, \dots, n \}$. Moreover, since $\Phi_{\varepsilon}$ is a finite set, there exists an index $\bar{\jmath}$ such that for any $k \in \{1, \dots, n \}$ we have that
	\begin{equation}
		|\varphi_k(x_{p_r}) - \varphi_k(x_{p_s})| < \varepsilon, \ \ \text{for all} \ r,s \geq \bar{\jmath}.
	\end{equation}
	We observe that $\bar{\jmath}$ does not depend on $\varphi$, but only on $\varepsilon$ and $\Phi_{\varepsilon}$.
	
	In order to prove that $(x_{p_{j}})_{j \in \mathbb{N}}$ is a Cauchy sequence in $X$ with respect to $D_X$, we observe that for any $r,s \in \mathbb{N}$ the following inequalities hold \textcolor{black}{for any $\varphi\in\Phi$}:
	\begin{align}
		|\varphi(x_{p_r}) - \varphi(x_{p_s})| &= |\varphi(x_{p_r}) - \varphi_{\bar k}(x_{p_r}) + \varphi_{\bar k}(x_{p_r}) - \varphi_{\bar k}(x_{p_s}) + \varphi_{\bar k}(x_{p_s}) - \varphi(x_{p_s})| \nonumber\\
		& \leq  |\varphi(x_{p_r}) - \varphi_{\bar k}(x_{p_r})| + |\varphi_{\bar k}(x_{p_r}) - \varphi_{\bar k}(x_{p_s})| + |\varphi_{\bar k}(x_{p_s}) - \varphi(x_{p_s})| \nonumber\\
		& \leq  \|\varphi - \varphi_{\bar k}\|_{\infty} + |\varphi_{\bar k}(x_{p_r}) - \varphi_{\bar k}(x_{p_s})| + \|\varphi_{\bar k} - \varphi\|_{\infty}.
	\end{align}
It follows that $|\varphi(x_{p_r}) - \varphi(x_{p_s})| < 3\varepsilon$ for every $\varphi \in \Phi$ and every $r,s \geq \bar{\jmath}$.
Thus, \textcolor{black}{$\max_{\varphi \in \Phi}|\varphi(x_{p_r}) - \varphi(x_{p_s})| = D_X(x_{p_r},x_{p_s}) \le 3\varepsilon$} for any $r,s \geq \bar{\jmath}$. Hence, the sequence $(x_{p_j})_{j \in \mathbb{N}}$ is a Cauchy sequence in $X$ with respect to $D_X$.
This proves that $X$ is totally bounded with respect to $D_X$. Since $\Delta_X\le D_X$, $X$ is also totally bounded with respect to $\Delta_X$.
\end{proof}

\begin{remark}
\label{remPhitotbounded}
The proof of Proposition \ref{thmXTOTBOUND} just relies on the assumption that $\Phi$ is totally bounded, without using the compactness of that space.
\end{remark}

\begin{corollary}
If $X$ is complete, then $X$ is compact.
\end{corollary}

\begin{proof}
It follows from Proposition \ref{thmXTOTBOUND}, by recalling that in pseudo-metric spaces a set is compact if and only if it is complete and totally bounded \cite{Ga64}.
\end{proof}

\subsection{$\Phi$-preserving homeomorphisms.}

We know that the elements of $\mathrm{Aut}_{\Phi}(X)$ are isometries with respect to $D_X$, and hence homeomorphisms with respect to $D_X$ (cf. \cite{BFGQ19}).
In the following, we show that each element of $G$ is also an isometry, and hence a homeomorphism, with respect to $\Delta_X$.

\begin{proposition}\label{pp1}
Each $g\in G$ is an isometry with respect to $\Delta_X$.
\end{proposition}

\begin{proof}
Let $\nu_g$ be the Borel measure on $V$ defined by setting \textcolor{black}{$\nu_g(A)=\nu(\rho_g^{-1}(A))$} for any Borel set $A$ in $V$ (recall that $\rho_g$ and $\rho_g^{-1}$ take Borel sets to Borel sets).
From the invariance of $\nu$ under the action of $G$, $\nu_g=\nu$.
By applying a change of variable, the invariance of $f$ under the action of each $g\in G$ implies that
\textcolor{black}{\begin{align*}
\Delta_X(g(x_1),g(x_2))&=\int_{\Phi}|\varphi g(x_1)-\varphi  g(x_2)|f(\varphi)\,d\nu\\
&=\int_{\Phi}|\varphi g(x_1)-\varphi  g(x_2)|f(\varphi g)\,d\nu\\
&=\int_{\Phi}|\varphi(x_1)-\varphi(x_2)|f(\varphi)\,d\nu_g\\
&=\int_{\Phi}|\varphi(x_1)-\varphi(x_2)|f(\varphi)\,d\nu=\Delta_X(x_1,x_2)
\end{align*}}
for any $x_1,x_2\in X$.
\end{proof}

Now, we turn $G$ into a pseudo-metric space by using \textcolor{black}{$\nu$ and $\|\cdot\|_V$.} To do this, we need the following lemma.

\begin{lemma}\label{lemmaxig1g2}
Let $g_1,g_2\in G$. Then, the map $\xi_{g_1,g_2}:\Phi\to\mathbb{R}$ sending $\varphi$ to $\lVert\varphi  g_1-\varphi  g_2\rVert_V$ is a continuous map with respect to \textcolor{black}{the topology $\tau_\Phi$ on $\Phi$ induced by the topology $\tau_V$ on $V$,} and it is an integrable random variable with respect to $\nu$.
\end{lemma}
\begin{proof}
Let $\varphi_1,\varphi_2\in\Phi$. Since $\lVert \cdot \rVert_V$ is invariant {under the action of $G$}, we have that

\begin{align*}
\left|\xi_{g_1,g_2}(\varphi_1)-\xi_{g_1,g_2}(\varphi_2)\right|&=\left|\lVert\varphi_1 g_1-\varphi_1 g_2\rVert_V-\lVert\varphi_2  g_1-\varphi_2  g_2\rVert_V\right|\\
&\le\lVert(\varphi_1  g_1-\varphi_1  g_2)-(\varphi_2  g_1-\varphi_2  g_2)\rVert_V\\
&=\lVert(\varphi_1  g_1-\varphi_2  g_1)-(\varphi_1  g_2-\varphi_2  g_2)\rVert_V\\
&\le\lVert\varphi_1  g_1-\varphi_2  g_1\rVert_V+\lVert\varphi_1 g_2-\varphi_2  g_2)\rVert_V=2\lVert\varphi_1-\varphi_2\rVert_V.
\end{align*}
Therefore, $\xi_{g_1,g_2}$
\textcolor{black}{
is continuous with respect to the topology $\tau_\Phi$ induced by the topology $\tau_V$ defined in point (\ref{coeffalpha}) of Subsection \ref{frame}.
Now, since $\xi_{g_1,g_2}$ is continuous on the compact space $\Phi$, it is bounded. \textcolor{black}{Moreover, the counterimage $\xi_{x_1,x_2}^{-1}(U)$ of each open set $U\subseteq \R$ is an open set (and hence a Borel set) in $\Phi$.
It follows that the counterimage of each Borel set is a Borel set. This proves that the function $\xi_{g_1,g_2}$ is measurable, since the restriction of $\nu$ to $\Phi$ is still a Borel measure.
Hence, recalling that the measure $\nu$ is finite}, if $M$ is \textcolor{black}{an upper bound for $|\xi_{g_1,g_2}|$}, then we have that
\[\int_{\Phi}\left|\xi_{g_1,g_2}(\varphi)\right|\,d\nu\le M\nu(\Phi)<\infty.\] }
\end{proof}

Now, we can define the following pseudo-metric on $G$:
$$\Delta_G(g_1,g_2)=\int_{\Phi}\lVert\varphi  g_1-\varphi  g_2\rVert_V f(\varphi)\,d\nu,\quad \forall\,g_1,g_2\in G.$$


In \cite{BFGQ19} a different pseudo-metric $D_G$ on $G$ has been considered, defined by setting
$D_G(g_1,g_2):= \max_{\varphi \in \Phi}{\lVert\varphi g_1-\varphi g_2\rVert_{\infty}}$.
The definition of $D_G$ is well-posed, since $\Phi$ is compact.
We observe that
\begin{align*}
 \Delta_G(g_1,g_2)
 &=\int_{\Phi}\lVert\varphi  g_1-\varphi  g_2\rVert_V f(\varphi)\,d\nu \\
 &\le \int_{\Phi}\beta\lVert\varphi  g_1-\varphi  g_2\rVert_\infty f(\varphi)\,d\nu \\
 &\le \int_{\Phi}\beta D_G(g_1,g_2) f(\varphi)\,d\nu\\
 &= \beta D_G(g_1,g_2)\int_{\Phi} f(\varphi)\,d\nu\\
 &= \beta D_G(g_1,g_2)
 \end{align*}
for every $g_1,g_2\in G$, \textcolor{black}{
where $\beta$ is the coefficient defined in point (\ref{coeffalpha}) of Subsection \ref{frame}.}
The inequality $\Delta_G\le \beta D_G$ implies that the topology $\tau_{D_G}$ induced by $D_G$ is finer than the topology $\tau_{\Delta_G}$ induced by $\Delta_G$.
We recall that if $(G, D_G)$ is complete, {then it is compact \cite[Theorem 3]{BFGQ19}.}

\textcolor{black}{For the remainder of the paper}, whenever not differently specified, we consider $G$ as a pseudo-metric space (and hence a topological space) with respect to $\Delta_G$.

\begin{lemma}\label{ipghjeruioh}
	Let $g_1,g_2,g_3 \in G$. We have that
	$$\Delta_G (g_1,g_2)=\Delta_G (g_1 g_3,g_2 g_3)= \Delta_G (g_3 g_1,g_3 g_2).$$
\end{lemma}
\begin{proof}
The first equality follows directly from the invariance of the norm $\lVert \cdot \rVert_V$ {under the action of $G$}.
Now we show that $\Delta_G (g_1,g_2)= \Delta_G (g_3 g_1,g_3 g_2)$. By applying a change of variable, the invariance of $\nu$ and $f$ under the action of $G$ implies that
\textcolor{black}{\begin{align*}
\Delta_G (g_3 g_1,g_3 g_2)
 &= \int_{\Phi} \lVert \varphi g_3 g_1 - \varphi g_3 g_2 \rVert_V f(\varphi) \,d\nu \\
 &= \int_{\Phi} \lVert \varphi g_3 g_1 - \varphi g_3 g_2 \rVert_V f(\varphi g_3) \,d\nu \\
 &= \int_{\Phi} \lVert \varphi g_1 - \varphi g_2 \rVert_V f(\varphi) \,d\nu_{g_3} \\
 &= \int_{\Phi} \lVert \varphi g_1 - \varphi g_2 \rVert_V f(\varphi) \,d\nu=\Delta_G(g_1,g_2)
\end{align*}}
where $\nu_{g_3}$ is the Borel measure on $V$ defined by setting \textcolor{black}{$\nu_{g_3}(A)=\nu(\rho_{g_3}^{-1}(A))$} for any Borel set $A$ in $V$ {\color{black}(recall that $\rho_g$ and $\rho_g^{-1}$ take Borel sets to Borel sets).}
\end{proof}

{\color{black}
\begin{remark}\label{remgaction}
Lemma \ref{ipghjeruioh} implies that the actions of $G$ on $G$ respectively taking $g$ to $g\bar g$ and $\bar g g$ are isometries for every $\bar g\in G$, with respect to $\Delta_{G}$.
\end{remark}
}

\begin{proposition}
The group $G$ is a topological group. Further, the action of $G$ on $\Phi$ by composition on the right is continuous.
\end{proposition}

\begin{proof}
First, we show that $G$ is a topological group. Let $\sigma:G\times G\to G$ and $\iota:G\to G$ be the composition and the inverse maps, respectively. We consider the product topology on $G\times G$. We must show that $\sigma$ and $\iota$ are continuous. To show that $\sigma$ is continuous, let $(g_1,g_2),(g'_1,g'_2)\in G \times G$. Using Lemma \ref{ipghjeruioh}, we have that
\begin{align*}
\Delta_G(g_1  g_2,g'_1 g'_2)&=\Delta_G (g_1,g'_1 g'_2 g_2^{-1})\\
 & \le \Delta_G (g_1,g'_1) + \Delta_G (g'_1,g'_1 g'_2 g_2^{-1}) \\
 & = \Delta_G (g_1,g'_1) + \Delta_G (\id_X, g'_2 g_2^{-1}) \\
 & = \Delta_G (g_1,g'_1) + \Delta_G (g_2, g'_2).
\end{align*}
It follows that the composition map $\sigma$ is continuous. Now, we show that $\iota$ is continuous. Consider $h_1,h_2 \in G$. We have that
\begin{align*}
\Delta_G(h_1^{-1},h_2^{-1})  &= \Delta_G(h_1^{-1}h_2,h_2^{-1}h_2)
\\& = \Delta_G(h_1^{-1}h_2,\id_X)
\\& =\Delta_G(h_1^{-1}h_2,h_1^{-1}h_1)
\\& =\Delta_G(h_2,h_1)
\\& =\Delta_G(h_1,h_2).
\end{align*}
This proves that $\iota$ is an isometry, and hence it is continuous.

Therefore, $G$ is a topological group.

Let us now assume that $\rho:\Phi\times G\to\Phi$ is {\color{black} the action of $G$ on $\Phi$ by composition on the right} (i.e., $\rho(\varphi,g)=\varphi g$ for any $\varphi\in\Phi$ and $g\in G$). We have to prove that $\rho$ is continuous, when $\Phi\times G$ is endowed with the product topology. Let $(\varphi_1,g_1),(\varphi_2,g_2)\in\Phi\times G$ and $\varepsilon>0$. Let us define ${B_\varepsilon}$ as the ball $B_{\Phi}(\varphi_2,\frac{\varepsilon}{4})$ in $\Phi$ with respect to $\lVert\cdot\rVert_V$. {\color{black} Since $\varphi_2$ is in the support of $\mu$}, $\mu({B_\varepsilon})$ is positive.
We will show that if
$\lVert\varphi_1-\varphi_2\rVert_V\le \frac{\varepsilon}{4}$ and
$\Delta_G(g_1,g_2)\le \frac{\varepsilon}{4}\mu(B_\varepsilon)$,
then $\lVert \rho(\varphi_1,g_1)- \rho(\varphi_2,g_2)\rVert\le\varepsilon$.

{\color{black} Let us assume that $\lVert\varphi_1-\varphi_2\rVert_V\le \frac{\varepsilon}{4}$ and
$\Delta_G(g_1,g_2)\le \frac{\varepsilon}{4}\mu(B_\varepsilon)$. Recalling the invariance of $\|\cdot\|_V$ under the action of $G$ and the inequality
$\lVert\varphi -\varphi_2\rVert_V\le \frac{\varepsilon}{4}$, holding for every $\varphi\in B_\varepsilon$,} we have that
\begin{align*}
\Delta_G(g_1,g_2)&=\int_{\Phi}\lVert\varphi g_1-\varphi  g_2\rVert_V f(\varphi)\,d\nu\\
&\ge\int_{{B_\varepsilon}}\lVert\varphi  g_1-\varphi  g_2\rVert_V f(\varphi)\,d\nu\\
&=\int_{{B_\varepsilon}}\lVert(\varphi  g_1-\varphi_2  g_1)+(\varphi_2  g_1-\varphi_2  g_2)+(\varphi_2  g_2-\varphi g_2)\rVert_V f(\varphi)\,d\nu\\
&\ge\int_{{B_\varepsilon}}\left(\lVert\varphi_2  g_1-\varphi_2  g_2\rVert_V-\lVert\varphi  g_1-\varphi_2 g_1\rVert_V-\lVert\varphi_2  g_2-\varphi  g_2\rVert_V\right) f(\varphi)\,d\nu\\
&= \int_{{B_\varepsilon}} \left(\lVert\varphi_2 g_1-\varphi_2  g_2\rVert_V-2\lVert\varphi -\varphi_2\rVert_V\right) f(\varphi)\,d\nu\\
&\ge \mu({B_\varepsilon}) \left(\lVert\varphi_2 g_1-\varphi_2  g_2\rVert_V-\frac{\varepsilon}{2}\right).
\end{align*}
It follows that
 $\lVert\varphi_2  g_1-\varphi_2 g_2\rVert_V\le\frac{\Delta_G(g_1,g_2)}{\mu({B_\varepsilon})}+\frac{\varepsilon}{2}$.
Therefore,
 \begin{align*}
 \lVert \rho(\varphi_1,g_1)-\rho(\varphi_2,g_2)\rVert_V &=\lVert\varphi_1 g_1-\varphi_2 g_2\rVert_V \\
 & \le\lVert\varphi_1 g_1-\varphi_2 g_1\rVert_V +\lVert\varphi_2 g_1-\varphi_2 g_2\rVert_V\\
 & \le \lVert\varphi_1-\varphi_2\rVert_V +\frac{\Delta_G(g_1,g_2)}{\mu({B_\varepsilon})}+\frac{\varepsilon}{2}\\
 &\le \frac{\varepsilon}{4}+\frac{\varepsilon}{4}+\frac{\varepsilon}{2}=\varepsilon.
 \end{align*}
Consequently, $\rho$ is continuous.
\end{proof}

In order to study the compactness of $G$, we need the following result.
\begin{proposition}
\label{thmGcompact}
$G$ is totally bounded with respect to $D_G$ and $\Delta_G$.
\end{proposition}

\begin{proof}
Let $(g_i)_{i\in \mathbb{N}}$ and $\varepsilon$ be a sequence in $G$ and a positive real number, respectively. Given that $\Phi$ is totally bounded, we can find a finite subset $\Phi_{\varepsilon} = \{ \varphi_1, \dots , \varphi_n \}$ such that if $\varphi\in\Phi$ then there exists $\varphi_{\bar k} \in \Phi_\varepsilon$ for which ${\lVert\varphi_{\bar k}-\varphi\rVert_{\infty}} < \varepsilon$.
	
Let us consider the sequence $(\varphi_1g_i)_{i\in \mathbb{N}}$ in $\Phi$. Since $\Phi$ is totally bounded,
it follows that we can extract a Cauchy subsequence $(\varphi_1g_{i_h})_{h\in \mathbb{N}}$. Then we consider the sequence $\left(\varphi_2g_{i_h}\right)_{h\in \mathbb{N}}$. Again, we can extract a Cauchy subsequence $\left(\varphi_2g_{i_{h_t}}\right)_{t\in \mathbb{N}}$. We can repeat the same argument for any $\varphi_k \in \Phi_{\varepsilon}$. Thus, we are able to extract a subsequence $(g_{p_j})_{j\in \mathbb{N}}$ of $(g_i)_{i\in \mathbb{N}}$ such that $(\varphi_k  g_{p_j})_{j\in \mathbb{N}}$ is a Cauchy sequence for any $k \in \{1, \dots , n\}$. For the finiteness of set $\Phi_\varepsilon$, we can find an index $\bar{\jmath}$ such that
for any $k\in \{1,\ldots,n\}$
	\begin{equation}
		{\lVert\varphi_k  g_{p_r}-\varphi_k g_{p_s}\rVert_{\infty}}< \varepsilon, \ \textnormal{for every} \ s,r \geq \bar{\jmath}.
	\end{equation}
We observe that $\bar{\jmath}$ does not depend on $\varphi$, but only on $\varepsilon$ and $\Phi_{\varepsilon}$.
	
In order to prove that $(g_{p_{j}})_{j\in \mathbb{N}}$ is a Cauchy sequence, we observe that for any $r,s \in \mathbb{N}$ we have
	{\begin{align}
			\lVert\varphi g_{p_r}-\varphi  g_{p_s}\rVert_{\infty}
		& \leq \lVert\varphi  g_{p_r}-\varphi_{\bar k} g_{p_r}\rVert_{\infty} + \lVert\varphi_{\bar k}  g_{p_r}-\varphi_{\bar k}  g_{p_s}\rVert_{\infty} + \lVert\varphi_{\bar k} g_{p_s}-\varphi g_{p_s}\rVert_{\infty}\nonumber \\
		& =  \lVert\varphi-\varphi_{\bar k}\rVert_{\infty} + \lVert\varphi_{\bar k} g_{p_r}-\varphi_{\bar k}  g_{p_s}\rVert_{\infty} + \lVert\varphi_{\bar k}-\varphi\rVert_{\infty}.
	\end{align}}
	
Since ${\lVert\varphi_{\bar k}-\varphi\rVert_{\infty}} < \varepsilon$, we get ${\lVert\varphi g_{p_r}-\varphi  g_{p_s}\rVert_{\infty}} < 3\varepsilon$ for every $r,s \geq \bar{\jmath}$.
Thus, {\color{black} the inequality $D_G(g_{p_r},g_{p_s})\le 3\varepsilon$ holds.} Hence, the sequence $(g_{p_j})_{j\in \mathbb{N}}$ is a Cauchy sequence with respect to $D_G$. Therefore, by Lemma \ref{TOTBOUNDLEM}, $G$ is totally bounded with respect to $D_G$.
\textcolor{black}{As we have previously seen, $\Delta_G\le \beta D_G$,
where $\beta$ is the coefficient defined in point (\ref{coeffalpha}) of Subsection \ref{frame}.
Hence, $G$ is also totally bounded with respect to $\Delta_G$.}
\end{proof}

Therefore, the following statement holds, by recalling that in pseudo-metric spaces a set is compact if and only if it is complete and totally bounded \cite{Ga64}.

\begin{corollary}
	If $G$ is complete, then $G$ is compact.
\end{corollary}

\subsection{Group equivariant operators.}\label{geosection}

Let $V_1\subseteq \mathbb{R}^{X_1}_b$ and $V_2\subseteq \mathbb{R}^{X_2}_b$ be \textcolor{black}{a finite dimensional probability inner product $G_1$-space and
a finite dimensional probability inner product $G_2$-space, respectively.}
Let $\langle \cdot , \cdot \rangle_{V_i}$, $\| \cdot\|_{V_i}$, $\nu_{V_i}$, $f_{V_i}$
be the inner product, the norm, the Borel measure and the probability density function considered on $V_i$,
for $i=1,2$ (all of them are $G_i$-invariant).{  Also, for $i=1,2$, we consider the probability Borel measure $\mu_{V_i}$ by setting  $\mu_{V_i}(A_i)=\int_{A_i}f_{V_i}(\phi)\;d\nu_{V_i}$ for any Borel  set $A_i\subseteq V_i$. We assume that the support of $\mu_{V_i}$ is compact, and  we denote it by $\Phi_i$, for $i=1,2$.}
Moreover, we recall that since $\lVert\cdot\rVert_{V_i}$ and $\lVert\cdot\rVert_{\infty}$ are equivalent in $V_i$, there exist two real numbers $\beta_i,\alpha_i>0$ such that
\begin{equation}
\label{equivmetV}
\alpha_i\lVert\cdot\rVert_{\infty}\le\lVert\cdot\rVert_{V_i}\le\beta_i\lVert\cdot\rVert_{\infty}
\end{equation}
for $i=1,2$.

Let us consider the Bochner space $L^2({\Phi_1},V_2)$ of all
square-integrable maps from ${\Phi_1}$ to $V_2$.
Explicitly:{\begin{equation}\label{eq:norm}
L^2({\Phi_1},V_2):=\left\{F \colon {\Phi_1} \to V_2\bigg| \int_{\Phi_1}\|F(\varphi)\|_{V_2}^2 f_{V_1}(\varphi) \;d\nu_{V_1} \mathrm{\ exists\ and\ is\ finite}\right\}.
\end{equation}
}
{\color{black} In particular, $L^2({\Phi_1},V_2)$ is a vector space.}
We define an inner product on it as follows:
\begin{equation}\label{eq:inner_product}
\langle F_1,F_2\rangle:=\int_{{\Phi_1}}\langle F_1(\varphi),F_2(\varphi)\rangle_{V_2}
f_{V_1}(\varphi)\,d\nu_{V_1}\quad\forall\,F_1,F_2\in L^2({\Phi_1},V_2).
\end{equation}

Let us select a homomorphism $T\colon {G_1} \to {G_2}$.
A \emph{group equivariant operator} (or simply a GEO) from $({\Phi_1},{G_1})$ to
$({\Phi_2},{G_2})$ is {\color{black} a
square-integrable map} $F \colon ({\Phi_1},\lVert \cdot \rVert_{V_1}) \to ({\Phi_2},\lVert \cdot \rVert_{V_2})$ satisfying the property $F(\varphi g) = F(\varphi)T(g)$ for any $\varphi$ in ${\Phi_1}$ and $g$ in ${G_1}$.
\begin{remark}
Since $\lVert \cdot \rVert_{V_1}$ and $\lVert \cdot \rVert_\infty$ are equivalent in $V_1$ and, $\lVert \cdot \rVert_{V_2}$ and $\lVert \cdot \rVert_\infty$ are equivalent in $V_2$, a GEO $F\colon ({\Phi_1},\lVert \cdot \rVert_{V_1}) \to ({\Phi_2},\lVert \cdot \rVert_{V_2})$ is
{\color{black} Borel measurable}
also with respect to the $L^\infty$-norm defined on ${\Phi_1}$ and ${\Phi_2}$.
\end{remark}
We define the following norms on the space of GEOs {from $({\Phi_1},{G_1})$ to
$({\Phi_2},{G_2})$}:

\begin{enumerate}
  \item {\color{black} $\vertiii{F}_\infty:= \sup_{\varphi \in {\Phi_1}}\lVert F(\varphi) \rVert_\infty$}
  \item {\color{black} $\vertiii{F}_{V_2}:= \sup_{\varphi \in {\Phi_1}}\lVert F(\varphi) \rVert_{V_2}$}
  \item $\vertiii{F}_{L^2}:= \left(\int_{\Phi_1}\lVert F(\varphi) \rVert_{V_2}^2 f_{V_1}(\varphi)\ d\nu_{V_1} \right)^{\frac{1}{2}}$.
\end{enumerate}

\textcolor{black}{Since $V_2$ is a finite dimensional real vector space, to show that $L^2(\Phi_1,V_2)$ is a Banach space, it is enough to show that $L^2(\Phi_1,\mathbb{R})$ is a Banach space. Note that the latter is a  well-known fact (for instance, see \cite[Chapter 4]{bogachev2007measure}). Clearly, $\vertiii{\cdot}_{L^2}$ is the norm corresponding to the inner product defined in Expression (\ref{eq:inner_product}). Hence,
 $L^2({\Phi_1},V_2)$ is \textcolor{black}{also} a Hilbert space.}
\begin{lemma}
	\label{topeq}
	$\vertiii{\cdot}_\infty$ and $\vertiii{\cdot}_{V_2}$ induce the same topology on the space of GEOs from $({\Phi_1},G_1)$ to $({\Phi_2},G_2)$.
\end{lemma}
\begin{proof}
\textcolor{black}{From (\ref{equivmetV}),}
we already know that for any $\varphi$ in ${\Phi_1}$
$$\alpha_2\lVert F(\varphi)\rVert_{\infty}\le\lVert F(\varphi)\rVert_{V_2}\le\beta_2\lVert F(\varphi) \rVert_{\infty}$$
\textcolor{black}{where $\alpha_2,\beta_2$ are two fixed positive coefficients.}
Hence, we have that:
{\color{black} $$\alpha_2\sup_{\varphi \in {\Phi_1}}\lVert F(\varphi)\rVert_{\infty}\le \sup_{\varphi \in {\Phi_1}} \lVert F(\varphi)\rVert_{V_2}\le\beta_2 \sup_{\varphi \in {\Phi_1}}\lVert F(\varphi) \rVert_{\infty}.$$}
Therefore, $\vertiii{\cdot}_\infty$ and $\vertiii{\cdot}_{V_2}$ are equivalent norms, and hence they induce the same topology on the space of GEOs from $({\Phi_1},G_1)$ to $({\Phi_2},G_2)$.
\end{proof}
\begin{lemma}
	\label{L2finerV2}
	The topology induced by $\vertiii{\cdot}_{V_2}$ on the space of GEOs from $({\Phi_1},G_1)$ to $({\Phi_2},G_2)$ is finer than the one induced by $\vertiii{\cdot}_{L^2}$.
\end{lemma}
\begin{proof}
Let us consider a GEO $F$ from $({\Phi_1},G_1)$ to $({\Phi_2},G_2)$. We have that
\begin{align*}
 \vertiii{F}_{L^2}
 &=\left(\int_{\Phi_1}\lVert F(\varphi) \rVert_{V_2}^2 f_{V_1}(\varphi)\ d\nu_{V_1}  \right)^{\frac{1}{2}}\\
 &\le{\color{black} \left(\int_{\Phi_1}\left(\sup_{\varphi \in {\Phi_1}}\lVert F(\varphi) \rVert_{V_2}^2\right) f_{V_1}(\varphi)\ d\nu_{V_1}  \right)^{\frac{1}{2}}}\\
 &=\vertiii{F}_{V_2}\left(\int_{\Phi_1} f_{V_1}(\varphi)\ d\nu_{V_1}  \right)^{\frac{1}{2}}\\
 &= \vertiii{F}_{V_2}.
 \end{align*}
	The above inequality directly implies the statement of the lemma.
\end{proof}

We can now define the concept of GENEO.

\begin{definition}\label{geneo}
A \emph{group equivariant non-expansive operator} (in brief, GENEO) {from $(({\Phi_1}, \lVert \cdot \rVert_{V_1}),G_1)$ to $(({\Phi_2}, \lVert \cdot \rVert_{V_2}),G_2)$} is a GEO $F$
from $({\Phi_1},G_1)$ to $({\Phi_2},G_2)$ such that
$$\lVert F(\varphi) - F(\varphi') \rVert_{V_2} \le \lVert \varphi - \varphi' \rVert_{V_1}$$
for every $\varphi, \varphi'\in\Phi_1$.
\end{definition}

{\begin{remark}
In \cite{BFGQ19}, another definition of a GENEO was given. Indeed, in the aforementioned paper, a GENEO is a map between two compact spaces of functions $\varPsi_1$ and $\varPsi_2$, which is group equivariant  with respect to a homomorphism $T: G_1\to G_2$ and non-expansive with respect to the $L^{\infty}$-norm. In the current framework, we call the aforementioned operators $L^{\infty}$-GENEOs. One could equip the space of $L^{\infty}$-GENEOs with the $L^{\infty}$-norm.
 It was shown in \cite{BFGQ19} that the space of all $L^{\infty}$-GENEOs from $({\varPsi_1},G_1)$ to $({\varPsi_2},G_2)$ is compact.
\end{remark}}

A key property of the space of GENEOs {in our framework}  is stated by the following theorem.

\begin{theorem}\label{teoremadicompattezza}
	The space of GENEOs {from $(({\Phi_1}, \lVert \cdot \rVert_{V_1}),G_1)$ to $(({\Phi_2}, \lVert \cdot \rVert_{V_2}),G_2)$} is compact with respect to the norm $\vertiii{\cdot}_{L^2}$.
\end{theorem}

\begin{proof}
{	Let  $(F_i)$ be a sequence of GENEOs from $(({\Phi_1}, \lVert \cdot \rVert_{V_1}),G_1)$ to $(({\Phi_2}, \lVert \cdot \rVert_{V_2}),G_2)$. It will suffice to prove that there exists a subsequence $(F_{i_j})$ of $(F_i)$ that converges in the $\vertiii{\cdot}_{L^2}$-topology.
%
From (\ref{equivmetV}) it follows that for every index $i$ and every $\varphi,\varphi'\in\Phi_1$
	\begin{equation}\label{eq1forreferee}
	    \alpha_2\lVert F_i(\varphi) - F_i(\varphi') \rVert_\infty \le \lVert F_i(\varphi) - F_i(\varphi') \rVert_{V_2} \le \lVert \varphi - \varphi' \rVert_{V_1} \le \beta_1 \lVert \varphi - \varphi' \rVert_\infty
	\end{equation}
	and hence
	$$ \lVert F_i(\varphi) - F_i(\varphi') \rVert_\infty \le \frac{\beta_1}{\alpha_2} \lVert \varphi - \varphi' \rVert_\infty.$$
	Therefore, $\left\| \frac{\alpha_2}{\beta_1}F_i(\varphi) - \frac{\alpha_2}{\beta_1}F_i(\varphi') \right\|_\infty \le  \lVert \varphi - \varphi' \rVert_\infty$}, {and each map
$\frac{\alpha_2}{\beta_1}F_i$ is an { $L^{\infty}$-GENEO} from $({\Phi_1},G_1)$ to $({\Phi_2'},G_2)$, where $\Phi_2'=\frac{\alpha_2}{\beta_1}\Phi_2$.
Since the space of { $L^{\infty}$-GENEO}s from $({\Phi_1},G_1)$ to $({\Phi_2'}, G_2)$ is compact with respect to $\vertiii{\cdot}_{\infty}$ (see Theorem 7 in \cite{BFGQ19}), we can consider the sequence $\left(\frac{\alpha_2}{\beta_1}F_i\right)$ and extract a subsequence $\left(\frac{\alpha_2}{\beta_1}F_{i_j}\right)$ that converges with respect to $\vertiii{\cdot}_{\infty}$ to an { $L^{\infty}$-GENEO} $\bar{F}$ from $({\Phi_1},G_1)$ to $({\Phi'_2},G_2)$.
The maps $\frac{\beta_1}{\alpha_2} \bar{F}$ and $(F_{i_j})$ are GEOs from $({\Phi_1},{G_1})$ to
$({\Phi_2},{G_2})$, and we observe that $(F_{i_j})$ converges to $\frac{\beta_1}{\alpha_2} \bar{F}$ with respect to $\vertiii{\cdot}_{\infty}$.}
{From Lemma \ref{topeq}, it follows that $(F_{i_j})$ converges to $\frac{\beta_1}{\alpha_2}\bar{F}$ \textcolor{black}{also} with respect to $\vertiii{\cdot}_{V_2}$.

Since $\lim_{j\to\infty}\sup_{\varphi \in {\Phi_1}}\left\| \frac{\beta_1}{\alpha_2} \bar{F}(\varphi)-F_{i_j}(\varphi) \right\|_{V_2}=0$, then $\frac{\beta_1}{\alpha_2} \bar{F}(\varphi)= \lim_{j \to \infty} F_{i_j}(\varphi)$ with respect to the norm $\lVert \cdot \rVert_{ V_2}$ (or, equivalently, the norm $\lVert \cdot \rVert_{\infty}$) for every $\varphi\in\Phi_1$.
Therefore, for any pair $\varphi_1,\varphi_2$ in ${\Phi_1}$
	\begin{align}
	\left\| \frac{\beta_1}{\alpha_2}\bar F(\varphi_1)-\frac{\beta_1}{\alpha_2}\bar F(\varphi_2)\right\|_{V_2} & = \left\lVert\lim_{j \to \infty}F_{i_j}(\varphi_1)-\lim_{j \to \infty}F_{i_j}(\varphi_2)\right\rVert_{V_2} \nonumber\\ \
	& = \lim_{j \to \infty}\lVert F_{i_j}(\varphi_1)-F_{i_j}(\varphi_2)\rVert_{V_2} \nonumber\\
	& \le \lim_{j \to \infty}\lVert \varphi_1-\varphi_2\rVert_{V_1} \label{eq2forreferee}\\
	& =\lVert \varphi_1-\varphi_2\rVert_{V_1}\nonumber.
	\end{align}
%
%
%
%
This proves that $\frac{\beta_1}{\alpha_2} \bar{F}$ is a GENEO from $({\Phi_1}, \lVert \cdot \rVert_{V_1})$ to $({\Phi_2}, \lVert \cdot \rVert_{V_2})$. Since $(F_{i_j})$ converges to $\frac{\beta_1}{\alpha_2}\bar{F}$ with respect to $\vertiii{\cdot}_{V_2}$, Lemma \ref{L2finerV2} guarantees that $(F_{i_j})$ converges to $\frac{\beta_1}{\alpha_2}\bar{F}$ \textcolor{black}{also} with respect to $\vertiii{\cdot}_{L^2}$.}
\end{proof}

{The compactness of the space of GENEOs with respect to the topology induced by the norm $\vertiii{\cdot}_{L^2}$ guarantees that such a space can be approximated by a finite set, as stated by the following result.

\begin{corollary}
	\label{thmapprox}
	For every $\varepsilon>0$, the space $\mathcal{F}^{\mathrm{all}}$ of all GENEOs from $(({\Phi_1}, \lVert \cdot \rVert_{V_1}),G_1)$ to $(({\Phi_2}, \lVert \cdot \rVert_{V_2}),G_2)$ admits a finite subset $\mathcal{F}$ such that for every $F\in\mathcal{F}^{\mathrm{all}}$
there exists an $F'\in \mathcal{F}$ such that $\vertiii{F-F'}_{L^2}<\varepsilon$.
\end{corollary}
\begin{proof}
	It immediately follows from Theorem \ref{teoremadicompattezza}, by considering a finite subcover of the open cover of $\mathcal{F}^{\mathrm{all}}$ whose elements are the balls of radius $\varepsilon$ centered at points of $\mathcal{F}^{\mathrm{all}}$.
\end{proof}

\begin{remark}
In the proof of Theorem \ref{teoremadicompattezza}, we used the ``non-expansive" assumption in the expressions (\ref{eq1forreferee}) and (\ref{eq2forreferee}).  Note that Theorem  \ref{teoremadicompattezza} no longer holds if we replace ``GENEOs" with ``GEOs". For instance, assume that
\begin{itemize}
    \item $X=[0,1]$, and $G=\{\id_X\}$;
    \item {\color{black} $V$ is the (1-dimensional) vector space of all real constant functions $\varphi_a:[0,1]\to\mathbb{R}$, with $\varphi_a(x)=a$ for all $x\in X$, equipped with the inner product $\langle\varphi_a,\varphi_b\rangle_V:=ab$};
    \item  {\color{black} $\Phi=\{\varphi_a:[0,1]\to\mathbb{R}: a\in[0,1]\}$.}
    \end{itemize}
 {\color{black} Here, $L^2(\Phi,V)$ is defined by setting $\Phi_1=\Phi\simeq [0,1]$, $V_1=V_2=V\simeq \R$, $\nu_{V_1}$ equal to the Lebesgue measure on $V\simeq\R$ and $f_{V_1}\equiv 1$ in the equality (\ref{eq:norm}).
    We note that $L^2(\Phi,V)$ and $L^2([0,1],\mathbb{R})$ are isomorphic vector spaces, where $[0,1]$ and $\mathbb{R}$ are both equipped with the absolute value norm. In other words, we can identify $\Phi$ with $[0,1]$ and $V$ with $\R$, by taking each function $\varphi_a$ to $a$. According to Definition 2.19, the space of all GEOs from $(\Phi,G)$ to $(\Phi,G)$
    with respect to $\mathrm{id}_G:G\to G$ is the space $L^2([0,1],[0,1])$ of all the square-integrable functions from $[0,1]$ to $[0,1]$.
    It is well-known that $L^2([0,1],[0,1])$ is not compact.}
 \end{remark}

Another important property of the space $\mathcal{F}^{\mathrm{all}}$ is stated by the next result.

\begin{proposition}
\label{propconvex}
If $\Phi_2$ is convex, then the set $\mathcal{F}^\mathrm{all}$ of all GENEOs from $(({\Phi_1}, \lVert \cdot \rVert_{V_1}),G_1)$ to $(({\Phi_2}, \lVert \cdot \rVert_{V_2}),G_2)$ is convex.
\end{proposition}
\begin{proof}

Let us consider $F, F'\in \mathcal{F}^\mathrm{all}$ and a value $t\in [0,1]$. We can define an operator $F_t \colon \Phi_1 \to \Phi_2$ by setting $F_t(\varphi) := (1-t)F(\varphi)+tF'(\varphi)$
for any $\varphi \in \Phi_1$.
Note that the convexity of $\Phi_2$ ensures us that $F_t$ is well defined.
First we prove that $F_t$ is a GEO from $(\Phi_1,G_1)$ to $(\Phi_2,G_2)$. Since $F$ and $F'$ are  equivariant, for every $\varphi\in\Phi_1$ and every $g\in G_1$
	$$F_t(\varphi  g) = (1-t)F(\varphi g)+tF'(\varphi g)=(1-t)F(\varphi)T(g)+tF'(\varphi)T(g)= F_t(\varphi)  T(g).$$	
	Since $F$ and $F'$ are non-expansive, $F_t$ is non-expansive:
	\begin{align*}
	\left\| F_t(\varphi)-F_t(\varphi')\right\|_{V_2} & = \left\|((1-t)F(\varphi)+tF'(\varphi)) - ((1-t)F(\varphi')+tF'(\varphi')) \right\|_{V_2}\\
	& =  \left\|(1-t)(F(\varphi)-F(\varphi')) + t(F'(\varphi)-F'(\varphi'))\right\|_{V_2} \\
	& \le  (1-t)\left\|F(\varphi)-F(\varphi')\right\|_{V_2} + t\left\|F'(\varphi)-F'(\varphi')\right\|_{V_2} \\
& \le  (1-t)\left\|\varphi-\varphi'\right\|_{V_1} + t\left\|\varphi-\varphi'\right\|_{V_1} \\
	& = \left\|\varphi-\varphi'\right\|_{V_1}.
	\end{align*}
	Therefore $F_t$ is a GENEO from $(({\Phi_1}, \lVert \cdot \rVert_{V_1}),G_1)$ to $(({\Phi_2}, \lVert \cdot \rVert_{V_2}),G_2)$.
\end{proof}
}

{\color{black}
\subsection{The action of $G$ on the space $\mathcal{F}^{\mathrm{all}}$ of all GENEOs.}\label{actiongeosection}

Let $\mathcal{F}^{\mathrm{all}}$ be the topological space of all GENEOs from $(({\Phi_1}, \lVert \cdot \rVert_{V_1}),G_1)$ to $(({\Phi_2}, \lVert \cdot \rVert_{V_2}),G_2)$, under the assumptions of the previous Section \ref{geosection}.
We already know that the group $G_i$ acts on ${\Phi_i}$, $X_i$, $G_i$ ($i=1,2$), respectively. For every $\bar g\in G_i$ these actions are indeed defined:
\begin{enumerate}
  \item $\varphi\mapsto \varphi \bar g$ for every $\varphi\in {\Phi_i}$;
  \item $x\mapsto \bar g(x)$ for every $x\in X_i$;
  \item $g\mapsto g\bar g$ and $g\mapsto \bar g g$ for every $g\in G_i$.
\end{enumerate}
Furthermore, we have already seen that these actions are isometries of $(\Phi_i,\Delta_{\Phi_i})$ (Subsect. \ref{frame}), $(X_i,\Delta_{X_i})$ (Prop. \ref{pp1}),
and $(G_i,\Delta_{G_i})$ (Remark \ref{remgaction}), respectively.

Now, we want to show that $G_1$ and $G_2$ also act isometrically on $\mathcal{F}^{\mathrm{all}}$.

For every GENEO $F$ from $(({\Phi_1}, \lVert \cdot \rVert_{V_1}),G_1)$ to $(({\Phi_2}, \lVert \cdot \rVert_{V_2}),G_2)$ with respect to $T:G_1\to G_2$, and every $\bar h\in G_2$, let us consider the map $F^{\bar h}:\Phi_1\to \Phi_2$ defined by setting $F^{\bar h}(\varphi):=F(\varphi)\bar h$, and
the homomorphism $T^{\bar h}:G_1\to G_2$ defined by setting $T^{\bar h}(g):=\bar h^{-1}T(g)\bar h$.

The following statement holds.

\begin{proposition}
\label{GactsonFall1}
For every $\bar h\in G_2$, the map $F^{\bar h}$ is a GENEO {\color{black} from $(({\Phi_1}, \lVert \cdot \rVert_{V_1}),G_1)$ to $(({\Phi_2}, \lVert \cdot \rVert_{V_2}),G_2)$} with respect to the homomorphism $T^{\bar h}$.
\end{proposition}
\begin{proof}
For every $\varphi,\varphi'\in \Phi_1$ and $g\in G_1$
\begin{align*}
	F^{\bar h}(\varphi g) & = F(\varphi g) \bar h\\
	& =  F(\varphi) T(g) \bar h \\
	& =  (F(\varphi) \bar h)  (\bar h^{-1} T(g) \bar h)  \\
    & = F^{\bar h}(\varphi) T^{\bar h}(g)
	\end{align*}
and
\begin{align*}
	\|F^{\bar h}(\varphi) - F^{\bar h}(\varphi')\|_\infty & = \|F(\varphi) \bar h - F(\varphi') \bar h\|_\infty\\
	& =  \|F(\varphi) - F(\varphi')\|_\infty \\
	& \le  \|\varphi - \varphi'\|_\infty.
	\end{align*}
\end{proof}

Proposition \ref{GactsonFall1} allows to define the group actions of $G_1$ and $G_2$ on $\mathcal{F}^{\mathrm{all}}$, respectively.

For any $\bar g\in G_1$ we can indeed consider the action that takes each $F\in \mathcal{F}^{\mathrm{all}}$ to the GENEO $F^{T(\bar g)}$, while
for any $\bar h\in G_2$ we can consider the action that takes each $F\in \mathcal{F}^{\mathrm{all}}$ to the GENEO $F^{\bar h}$.
The interesting point is that both these actions are isometries, as stated by the following proposition concerning the restriction
$\langle \cdot,\cdot\rangle_{\vert\mathcal{F}^{\mathrm{all}}}$ to
$\mathcal{F}^{\mathrm{all}}$ of the inner product $\langle \cdot,\cdot\rangle$ we have defined on $L^2({\Phi_1},V_2)$.

\begin{proposition}
\label{GactsonFall2}
For any $\bar g\in G_1$, the action that takes each $F\in \mathcal{F}^{\mathrm{all}}$ to the GENEO $F^{T(\bar g)}$ preserves $\langle \cdot,\cdot\rangle_{\vert\mathcal{F}^{\mathrm{all}}}$. For any $\bar h\in G_2$, the action that takes each $F\in \mathcal{F}^{\mathrm{all}}$ to the GENEO $F^{\bar h}$ preserves $\langle \cdot,\cdot\rangle_{\vert\mathcal{F}^{\mathrm{all}}}$.
\end{proposition}

\begin{proof}
For every $F_1,F_2\in \mathcal{F}^{\mathrm{all}}$
\begin{align*}
	\langle F_1^{T(\bar g)},F_2^{T(\bar g)}\rangle_{\vert\mathcal{F}^{\mathrm{all}}}
    & = \int_{{\Phi_1}}\langle F_1^{T(\bar g)}(\varphi),F_2^{T(\bar g)}(\varphi)\rangle_{V_2}f_{V_1}(\varphi)\,d\nu_{V_1}\\
	& =  \int_{{\Phi_1}}\langle F_1(\varphi),F_2(\varphi)\rangle_{V_2}f_{V_1}(\varphi)\,d\nu_{V_1}\\
	& =  \langle F_1,F_2\rangle
	\end{align*}
and
\begin{align*}
	\langle F_1^{\bar h},F_2^{\bar h)}\rangle_{\vert\mathcal{F}^{\mathrm{all}}}
    & = \int_{{\Phi_1}}\langle F_1^{\bar h}(\varphi),F_2^{\bar h}(\varphi)\rangle_{V_2}f_{V_1}(\varphi)\,d\nu_{V_1}\\
	& =  \int_{{\Phi_1}}\langle F_1(\varphi),F_2(\varphi)\rangle_{V_2}f_{V_1}(\varphi)\,d\nu_{V_1}\\
	& =  \langle F_1,F_2\rangle
	\end{align*}
since, {\color{black} as stated at the beginning of Subsection \ref{geosection},}  we are assuming that $\langle \cdot,\cdot\rangle_{V_2}$ is invariant under the action of $G_2$.
\end{proof}
}

\subsection{Submanifolds of GENEOs.}
\label{submanifolds}
In this subsection we discuss how it is possible to define a Riemannian structure on a manifold of GENEOs. We also briefly recall some basic definitions and results about Hilbert manifolds and Riemannian manifolds. For further details and the general theory, see \cite{abr12,Klin11}.

Since $L^2(\Phi, V_2)$ is a Hilbert space, it has a natural structure of Hilbert manifold with a single global chart given by the identity function on $L^2(\Phi, V_2)$. Each tangent space $T_p L^2(\Phi, V_2)$ at any point $p \in L^2(\Phi, V_2)$ is canonically isomorphic to $L^2(\Phi, V_2)$ itself.
We can give a Riemannian structure on $L^2(\Phi, V_2)$ by defining a metric $g$ as $g(v,w)(p) = \langle v , w \rangle$ for every $v,w \in  T_p L^2(\Phi, V_2)$, where $\langle \cdot , \cdot \rangle$ is the inner product on $L^2(\Phi, V_2)$. Now, let $M$ be a $C^k$-submanifold of $L^2(\Phi, V_2)$, such that each element of $M$ is a GENEO. Hence, $M$ naturally inherits a Riemannian structure from $L^2(\Phi, V_2)$.
In real applications we are often interested in finite dimensional manifolds. For this reason, we give an explicit definition. A subset $M$ is called an $m$-dimensional $C^k$-submanifold of $L^2(\Phi, V_2)$, if for all $p \in M$ there exist a neighborhood $U$ of $p$ in $L^2(\Phi, V_2)$, an open set $O \subseteq \mathbb{R}^m$ and a $C^k$-map $\phi \colon O \to L^2(\Phi, V_2)$, such that
\begin{enumerate}
	\item $\phi \colon O \to U \cap M$ is a homeomorphism;
	\item the differential $D \phi(x)$ is injective for every $ x \in O$.

\end{enumerate}
We call $\phi$ a parametrization in $p$. In this case, we have an explicit formula for the tangent space $T_p M$ depending on $\phi$:
$$T_p M = D\phi(x) \mathbb{R}^m, \ x= \phi^{-1}(p).$$

{\color{black} \section{An application of our model to select GENEOs}}
\label{Experiments}

In this section, {we will assume that $X$ is the torus $S^1\times S^1$.} The usual embedding of $S^1$ in $\R^2$ entails that we can endow $X$ with the metric induced by the {natural} immersion in $\R^4$. Given a positive integer $n$, we define the finite group $G$ whose elements are the isometries of $X$ that preserve the set $Y=\left\{\left(\frac{2\pi i}{n},\frac{2\pi j}{n}\right): i,j \in \{0,\ldots,n-1\}\right\}\subseteq S^1\times S^1$. Moreover, we consider the function $\chi: X \to\R$ as in the following:

\begin{align}\label{eq:chi}
\chi (x,y):= \begin{cases}
 1 & (x,y) \in \gamma \times \gamma \\
 0 &  \text{otherwise}
 \end{cases}
\end{align}

where $\gamma \subset S^{1}$ is the arc connecting $-\frac{\pi}{n}$ and $\frac{\pi}{n}$ counterclockwise.

We assume $V_1,V_2$ both equal to the $n^2$-dimensional vector space $V$ whose elements are the functions $\varphi:X\to\R$ that can be expressed as:
\begin{equation}\label{eq:phi}
    \varphi(\alpha,\beta):=\sum_{i=0}^{n-1}\sum_{j=0}^{n-1}a_i^j \chi\left(\alpha-\frac{2\pi i}{n},\beta-\frac{2\pi j}{n}\right),
\end{equation}
where $a_{i}^{j}\in\R$ for every $i,j \in \{0,\ldots,n-1\}$.
{\color{black}
After identifying $V$ with $\R^{n^2}$, we endow $V$ with the standard inner product.
}
{\color{black} Moreover, according to the properties required in Section \ref{geosection}, we set $\mu_V(E):=\int_E f_V\ d\nu_V$,
where $f_{V}$ will be defined later and $\nu_V$ is the counting measure on a collection $\Phi$ of discretized images
representing handwritten letters of the English alphabet sampled from the EMNIST dataset.}

Our manifold $M$ of GENEOs is the one containing every GENEO $F$ defined by setting:

\begin{equation}\label{eq:operatore_F}
F(\varphi)(\alpha,\beta):=\frac{1}{4\sqrt{m}}\left(\sum_{t=1}^{m}\sum_{r,s\in\{-1,1\}}u_t\varphi\left(\alpha+r\frac{k_t\pi}{n}, \beta+s\frac{h_t\pi}{n}\right)\right),
\end{equation}

where $m$ is a {fixed} positive integer, $u=(u_1, \dots, u_{m})\in S^{m-1}\subset\R^{m}$, i.e., $\sum_{t=1}^{m}u_t^2=1$, and $k=(k_1, \dots, k_m),h=(h_1, \dots, h_m)$ {are two fixed $m$-tuples of positive integers.} Therefore, each element in $M$ is identified by {the unit vector $u$.} We remark that the scaling factor $\frac{1}{4\sqrt{m}}$ in \eqref{eq:operatore_F} ensures that the operator $F$ is non-expansive.

We now move to a discrete setting by thinking of $V$ as a set of 2D discrete images {\color{black} on a discretized torus}. Given $\varphi \in V$, we refer to a 2D discrete image as the restriction of $\varphi$ to the subset $Y \subset X$, that is $\varphi_{|_Y}: Y \to \R$. {By abuse of notation, to improve the readability of the following lines we avoid the restriction symbol and denote by $\varphi$ a generic image}.  More precisely, given the definitions \eqref{eq:chi}, \eqref{eq:phi} and assuming $Y$ is a discrete $n \times n$ grid in $X$, the generic element $\varphi$ is a discrete 2D image whose $(i,j)$-pixel has value {$\sum_{r=0}^{n-1}\sum_{s=0}^{n-1}a_i^j \chi\left(\frac{2\pi (i-r)}{n},\frac{2\pi (j-s)}{n}\right)$}. We refer to the point $\left(\frac{2\pi i}{n},\frac{2\pi j}{n}\right) \in Y$ as $(i,j)$-pixel for all $i,j=0,\dots,n-1$.  Since $Y$ is preserved by the finite group of isometries $G$, we assume periodic boundary conditions for each generic 2D image. Finally, we suppose the support of the density $f_{V}$ is contained in the EMNIST dataset \cite{cohen2017emnist}, therefore $\Phi$ is defined as a finite subset of the EMNIST of cardinality $N$.

Obviously, the cardinality of the space of GENEOs $M$ makes unfeasible a detailed analysis of these operators. In particular, it can be desirable to restrict this space to a finite set of GENEOs which fairly approximates the whole manifold. In order to simplify the problem and to reduce the degrees of freedom we fix the vectors $k \in \N^{m}$ and $h \in \N^{m}$. Consequently, the manifold $M$ has dimension $m-1$ and each element in $M$ is now identified by the vector $u \in S^{m-1}$.

We now propose a procedure to extract a finite set $M_{r} \subset M$ of $r$ GENEOs, where $r$ is a positive integer. The computation of $M_{r}$ must be guided by a metric which takes into account the observation that in a generic application not all the admissible signals in $\Phi$ can be considered to be equivalent. For example, in the English alphabet we can think the higher is the frequency of a letter, which is simply the amount of times it appears on average in written language, the higher is its importance. In this regard, we remark that the manifold $M$ is endowed with the Riemannian structure defined by considering the inner product \eqref{eq:inner_product} which allows us to provide $M$ with the norm $\vertiii{\cdot}_{L^2}$. The latter takes into account the previous observation by weighting the contribution of each of the admissible signals in $\Phi$ through a probability density function.

The basic idea of our proposal is to cast the problem of defining $M_{r}$ as a non-linear constrained optimization problem whose objective is $L^2$-norm based and the constraints reflect the hypothesis $u \in S^{m-1}$. In particular, each GENEO in $M$ is identified through a vector $u \in S^{m-1}$, therefore seeking for $r$ operators in $M$ is equivalent to search for $r$ vectors in $S^{m-1}$. Then, we propose to read the task of defining $M_{r}$ as the following constrained optimization problem:

\begin{equation}\label{eq:minimization_problem}
    \underset{u_{i} \in S^{m-1}, \forall i=1 \dots r}{\arg \min}  \sum_{i=1}^{r} \sum_{j=i+1}^{r} \dfrac{1}{\vertiii{ F_{u_{i}} - F_{u_{j}}}^{2}_{L^2}}.
\end{equation}

By definition \eqref{eq:norm}, the optimization problem in \eqref{eq:minimization_problem} is equivalent to the following one:

\begin{equation}\label{eq:minimization_problem2}
    \underset{u_{i} \in S^{m-1}, \forall i=1 \dots r}{\arg \min}  \sum_{i=1}^{r} \sum_{j=i+1}^{r} \dfrac{1}{\sum_{s=1}^{N} f_{V}({\varphi_s}) \rVert F_{u_{i}}({\varphi_{s}}) - F_{u_{j}}({\varphi_{s}}) \rVert_{V}^{2}}.
\end{equation}

In our example we assume $\Phi$ is constituted by $N=26$ different images representing the handwritten letters of the English alphabet sampled from the EMNIST dataset. We now consider $\varphi_{s}$ as the $s$-th letter in $\Phi$. Then, the probability density $f_{V}$ on $\varphi_s$ is defined as the frequency of the letter $\varphi_s$ in the English alphabet \cite{lewand2000cryptological}. In addition, we suppose that $f_{V}(\varphi) = 0$ for all $\varphi \in V \setminus \Phi$.

We stress that the objective in \eqref{eq:minimization_problem2} is non-smooth and non-convex. A feasible (satisfying the constraints) local-minimizer is seeked through the iterative \textit{interior point} approach which is performed by using the built-in Matlab function, namely \texttt{fmincon}. The interior point approach used, at each iteration computes a feasible solution following the gradient direction of the objective \cite{nocedal2006numerical} starting from a given feasible initial guess.

In our experiment, the iterates performed by \texttt{fmincon} are stopped when the optimality measure and the constraints tolerance are lower than $10^{-6}$ \cite{nocedal2006numerical}. In addition, we have fixed a maximum number of objective function evaluations equals to 3000. We refer to $M_{r}^{k}$ as the $k$-th iterate of the interior point method and, in particular, by $M_{r}^{0}$ we refer to the starting iteration of the optimization process. Moreover, we fix $m=2$ and $r= 10, 20$.

In Fig. \ref{fig:m1} we represent by a blue circular line the manifold $M$. In Fig \ref{fig:m1} (a)-(c) we depict by red dots the starting iterates $M^{0}_{r}$ of the optimization process when $r=10$ and $r=20$, respectively, whereas in Fig \ref{fig:m1} (b)-(d) the red dots represent the solutions $M_{10}$ and $M_{20}$ of \eqref{eq:minimization_problem2} when $M^{0}_{10}$ and $M^{0}_{20}$ are chosen as initial guess, respectively. We observe that following the gradient flow of the objective in \eqref{eq:minimization_problem2} the outcomes $M_{10}$ and $M_{20}$ present GENEOs (red points) better distributed and separated than the ones in $M^{0}_{10}$ and $M^{0}_{20}$ which have been randomly selected. Indeed, the objective has been defined to penalize mostly close (with respect to the given $L^2$-metric) operators.

\begin{figure}[ht!]
	\centering
	\subfloat[$M^{0}_{10}$]{{\includegraphics[width=0.5\textwidth]{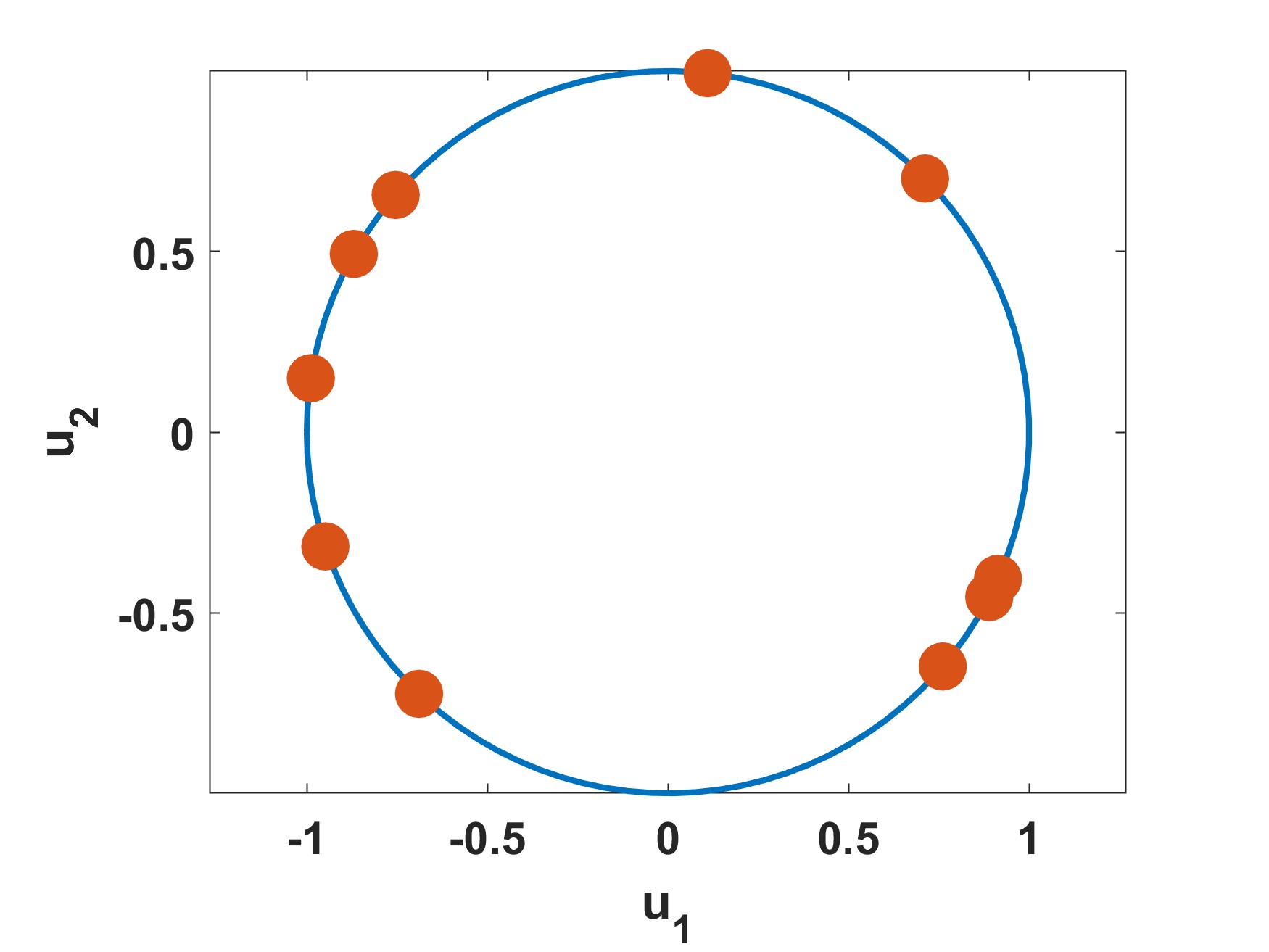}}}
	\subfloat[$M_{10}$]{{\includegraphics[width=0.5\textwidth]{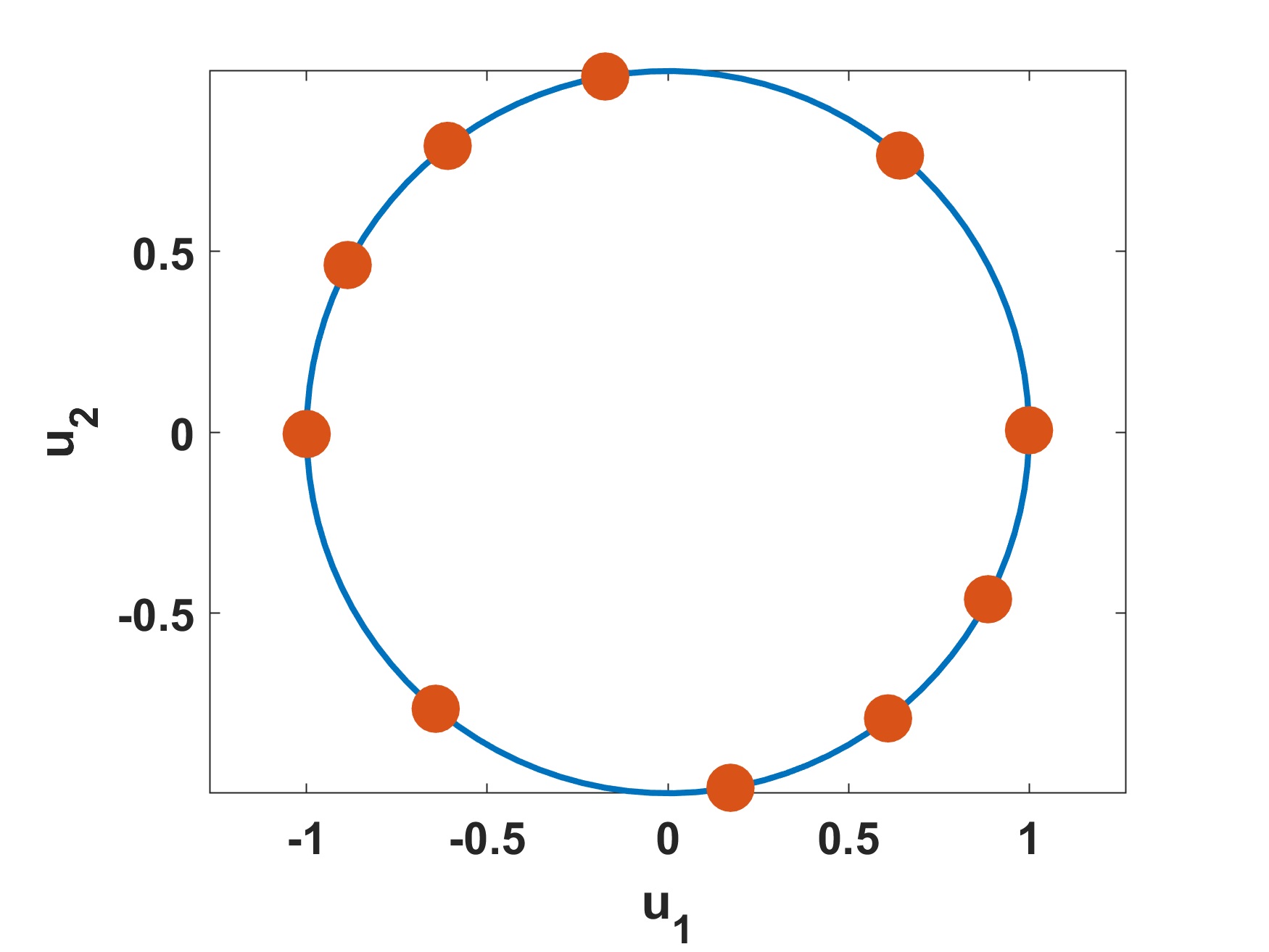}}}
	
	\subfloat[$M^{0}_{20}$]{{\includegraphics[width=0.5\textwidth]{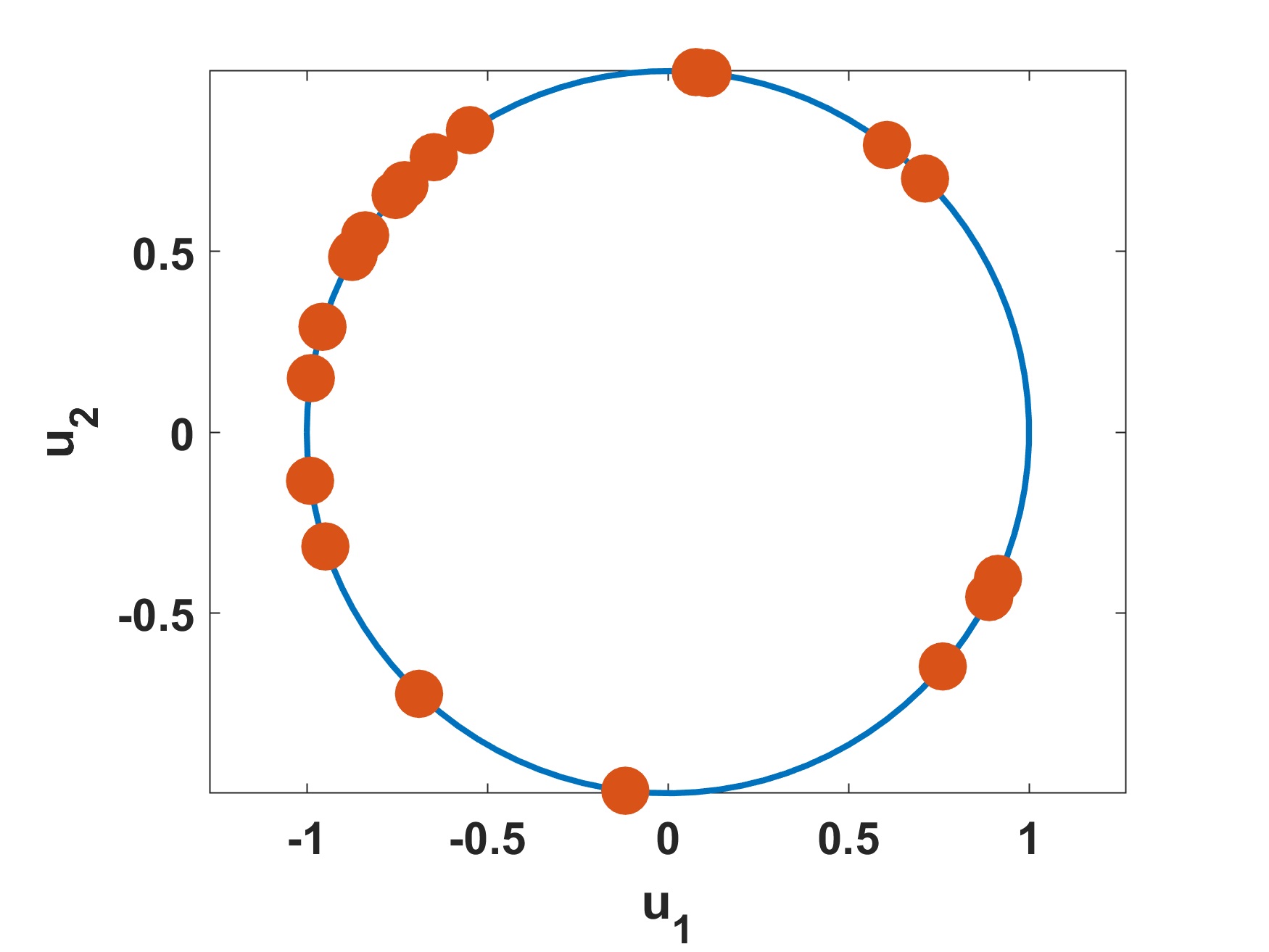}}}
	\subfloat[$M_{20}$]{{\includegraphics[width=0.5\textwidth]{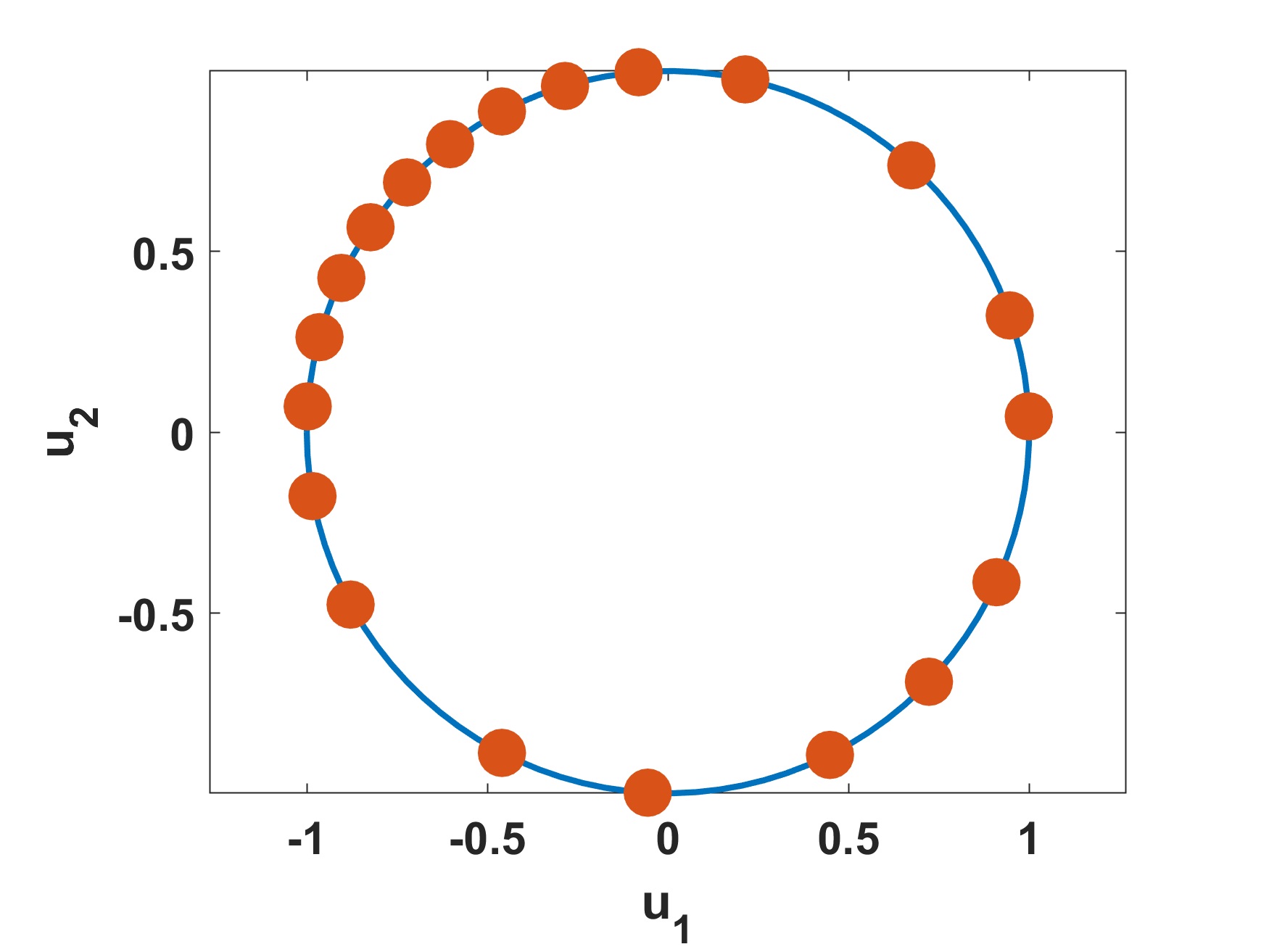}}}
	\caption{Random initial guess and outcomes of the interior-point method solving the optimization problem \eqref{eq:minimization_problem2} when $m=2$. (A)-(B) starting feasible iterate and related solution given $r=10$. (C)-(D) starting feasible iterate and related solution given {$r=20$}.}
	\label{fig:m1}
\end{figure}

We now want to evaluate how much better than the random $M^{0}_{10}$ and $M^{0}_{20}$ the computed $M_{10}$ and $M_{20}$ approximate the manifold $M$. For this purpose, we consider the evaluation set $E \subset M$ made up of 100 GENEOs belonging to $M$ randomly selected, then for every $F \in E$ we define $\eta_{r}(F)$ as the following ratio:
\begin{equation} \label{eq:eta}
    \eta_{r}(F):= \dfrac{\underset{F_{i}\in M_{r}}{\text{min}}  \vertiii{F - F_{i}}^{2}_{L^2}}{\underset{F_{j}\in M^{0}_{r}}{\text{min}}  \vertiii{F - F_{j}}^{2}_{L^2}}.
\end{equation}
We remark that $\eta_{r}(F)$, according to the formula \eqref{eq:eta}, measures the ratio of the minimum distances among the GENEO $F$ belonging to $E$ and the ones in $M_{r}$ and $M_{r}^{0}$. Therefore, a value $\eta_{r}(F)$ less than one entails that $F \in E$ is closer in terms of $L^2$-distance to a GENEO in $M_{r}$, conversely $F \in E$ is closer to a GENEO in $M^{0}_{r}$. Since $E$ is made up of 100 GENEOs we compute the mean values of $\left\{\eta_{r}(F) \right\}_{F \in E}$ choosing $r=10$ and $r=20$. In the first case the mean value is equal to $0.5833$ while in the second case it is equal to  $0.5442$. This states that, in mean, the outcomes $M_{10}$ and $M_{20}$ better approximate the whole manifold $M$ if compared to their random counterparts $M^{0}_{10}$ and $M^{0}_{20}$ used as initial guesses in the iterative scheme solving \eqref{eq:minimization_problem2} when $r=10$ and $r=20$, respectively.

%

\section{Discussion}
\label{Discussion}
In this paper we have shown how the space $\mathcal{F}$ of all GENEOs can be endowed with a suitable geometrical structure, so allowing to prove some of the relevant properties of such a space. In particular, we have devised a metric that guarantees the compactness of $\mathcal{F}$ under the main assumption that the set $\Phi$ of signals we are considering is compact. Such a metric differs from the one introduced in \cite{BFGQ19} and takes into account the probability of signals in $\Phi$. Our compactness result requires new metrics on $\Phi$ and the equivariance group $G$ and implies that $\mathcal{F}$ can be approximated by a finite set of operators within an arbitrarily small error. Furthermore, we have shown how we can define a suitable Riemannian structure on finite dimensional submanifolds of the space of GENEOs and take advantage of this structure to build approximations of those manifolds via a gradient flow minimization. We believe that the availability of these results could be of use in reducing the complexity of selecting suitable operators in applications concerning topological data analysis and deep learning. In the future, we would like to explore the possibility of extending our results to operators that are equivariant and non-expansive only almost everywhere on $\Phi$, so improving the theoretical and practical use of our mathematical model. {Another line of research could be reconstructing the theory in \cite{BFGQ19} by considering another metric on the space $\Phi$ instead of the one induced by the $L^{\infty}$-norm. In this way, new practical results could be compared with the ones in \cite{BFGQ19} to realize which metrics are efficient in different experiments. }

\section*{Acknowledgement}
This research has been partially supported by INdAM-GNSAGA and INdAM-GNCS. {The research of A.S. has been supported by the Institute for Research in Fundamental Sciences (IPM).} The authors thank Elena Loli Piccolomini for her helpful advice. {The authors would like to thank Nikolaos Chalmoukis for comments that greatly improved the manuscript.} This paper is dedicated to the memory of Ivana Paoletti, Giorgio Mati and Don Piero Vannelli.

\section*{Conflict of interest}
The authors declare that they have no conflict of interest.

\section*{Authors' contributions} P.F. devised the project. P.F., N.Q. and A.S. developed the mathematical model. P.C. took care of the experiments. All authors read and approved the final manuscript. The authors of this paper have been listed in alphabetical order.

\bibliographystyle{amsplain}
\bibliography{Riemannian_paper_bib}

\end{document}